\documentclass{amsart}
\usepackage{amssymb}
\usepackage{amsthm}
\usepackage{color}
\usepackage{graphicx}
\usepackage{latexsym}
\usepackage{url}
\usepackage[cmtip,all]{xy}

\newtheorem{theorem}{Theorem}[section]
\newtheorem{lemma}[theorem]{Lemma}
\newtheorem{proposition}[theorem]{Proposition}

\newtheorem{conjecture}[theorem]{Conjecture}
\newtheorem{corollary}[theorem]{Corollary}

\theoremstyle{definition}
\newtheorem{definition}[theorem]{Definition}
\newtheorem{example}[theorem]{Example}

\theoremstyle{remark}

\newtheorem{openproblem}[theorem]{Open Problem}
\newtheorem{question}[theorem]{Question}

\numberwithin{equation}{section}

\newcommand{\dist}{\operatorname{dist}}
\newcommand{\rank}{\operatorname{rank}}
\newcommand{\supp}{\operatorname{supp}}
\newcommand{\hamm}{\operatorname{hamm}}

\def\N{\mathbb{N}}

\def\R{\mathbb{R}}
\def\Z{\mathbb{Z}}

\begin{document}

\title[Combinatorics and Geometry of Transportation Polytopes]{Combinatorics and Geometry of Transportation Polytopes: An Update}

\author{Jes\'us~A.~De~Loera}
\author{Edward~D.~Kim}


\date{}

\begin{abstract}
A transportation polytope consists of all multidimensional arrays or tables of non-negative real numbers that satisfy certain sum conditions on subsets of the entries. They arise naturally in optimization and statistics, and also have interest for discrete mathematics because permutation matrices, latin squares, and magic squares appear naturally as lattice points of these polytopes.
                                                                                
In this paper we survey advances on the understanding of the combinatorics and geometry of these polyhedra and include some recent unpublished results on the diameter of graphs of these polytopes. In particular, this is a thirty-year update on the status of a list of open 
questions last visited in the 1984 book by Yemelichev, Kovalev and Kravtsov and the 1986 survey paper of Vlach. 
\end{abstract}

\maketitle

\section{Introduction}

Transportation polytopes are well-known objects in mathematical programming and statistics. 
In the operations research literature, classical transportation problems arise from the problem 
of transporting goods from a set of factories, each with given supply outcome, and a set of consumer centers, 
each with an amount of demand. Assuming the total supply equals the total demand
and that costs are specified for each possible pair (factory, consumer center), one may wish to optimize the cost of transporting goods. 
Indeed this was the original motivation that led Kantorovich (see~\cite{Kantorovich:On-the-translocation-of-masses}), Hitchcock (see~\cite{Hitchcock:The-Distribution-of-a-Product}), and T.~C.~Koopmans (see~\cite{Koopmans:OptimumTransportation}) to look at these problems. They are indeed among the first linear programming problems investigated, and Koopmans received the Nobel Prize in Economics for his work in this area (see~\cite{hoffman} for an interesting historical perspective). Not much later Birkhoff (see~\cite{Birkhoff:TresObservaciones}), von Neumann (see~\cite{Neumann:A-certain-zero-sum}), and Motzkin (see, e.{}g.{},~\cite{Motzkin:TransportationProblem}) were key contributors to the topic. 
The success of combinatorial algorithms such as the Hungarian method (see~\cite{Balinski:A-primal-method, Bourgeois:An-Extension-of-the-Munkres, Edmonds:Paths-Trees, Edmonds:improvementsAlgorithm, Klafszky:Variants-of-the-Hungarian, Kuhn:The-Hungarian-method, Kuhn:Variants-of-the-Hungarian, Munkres:AlgorithmsAssignmentTransportation, Tomizawa:n3}) depends on the rich combinatorial structure of the convex polyhedra
that defined the possible solutions, the so called \emph{transportation polytopes}.

In statistics, people have looked at the \emph{integral} transportation tables, which are widely known as \emph{contingency tables}. In 
statistics, a contingency table represents sample data arranged or tabulated by categories of combined  properties. Several questions motivate the study of the geometry of contingency tables, for instance,  in the table entry security problem: given a table $T$ (multi-dimensional perhaps) 
with statistics on private data about individuals, we may wish to release aggregated marginals of such a table without
disclosing information about the exact entries of the table. What can a data thief discover about $T$ from the published marginals?
When is $T$ uniquely identifiable by its margins? This problem has been studied by many researchers (see~\cite{Buzzigoli:LowerUpper,  Chowdhury:Disclosure-detection, Cox:ContingencyBounds, Cox:PropertiesStatistical, Dobra:Bounds-for-cell, Duncan:Disclosure, Duncan:DisclosureMultiple, Fienberg:Frechet-and-Bonferroni, IJ} and the references therein). Another natural problem is whether a given table presents strong evidence of significant relations between the characteristics tabulated (e.g., is cancer related to smoking). There is a lot of interest among statisticians on testing significance of independence for variables. Some methods depend on counting all possible contingency tables with given margins (see e.g., \cite{DG,Mehta:NetworkAlgorithmFisher}). This in turn is an interesting combinatorial geometric problem on the lattice points of transportation polytopes.

In this article we survey the state of the art in the combinatorics and geometry of transportation polytopes and contingency tables. The survey~\cite{Vlach:SolutionsPlanarTransportation} by Vlach, the 1984 monograph~\cite{YKK} by Yemelichev, Kovalev, and Kravtsov, and the paper~\cite{Klee:FacesTransportation} by Klee and Witzgall summarized the status of transportation polytopes up to the 1980s. Due to recent advances on the topic by the authors and others, we decided to write a new updated survey collecting remaining open problems and presenting recent solutions. We also included details on some unpublished new work on the diameter of the graphs of these polytopes.

In what follows we will denote by $[q]=\{1,2,\dots,q\}$. Similarly $\R^n_{\geq 0}$ denotes those vectors in $\R^n$ whose entries
are non-negative. Our notation and terminology on polytopes follows~\cite{Grunbaum:Polytopes} and~\cite{Ziegler:Lectures}.

\section{Classical transportation polytopes ($2$-ways)}\label{section:classical}

We begin by introducing the most well-known subfamily, the classical  transportation polytopes in just two indices. We call them \emph{$2$-way} transportation polytopes and in general $d$-ways refers to the case of variables with $d$ indices. Many of these facts are well-known and can be found in \cite{YKK}, but we repeat them here as we will use them in what follows.

Fix two integers $p, q \in \Z_{> 0}$. The \emph{transportation polytope} $P$ of size $p \times q$ defined by the vectors $u \in \R^p$ and $v \in \R^q$ is the convex polytope defined in the $pq$ variables $x_{i,j} \in \R_{\geq 0}$ ($i \in [p], j \in [q]$) satisfying the $p+q$ equations
\begin{equation}\label{equation:classicalsums}
\sum_{j=1}^q x_{i,j} = u_i\ (i \in [p])
\quad\text{and}\quad
\sum_{i=1}^p x_{i,j} = v_j\ (j \in [q]).
\end{equation}
Since the coordinates $x_{i,j}$ of $P$ are non-negative, the conditions \eqref{equation:classicalsums} imply $P$ is bounded. The vectors $u$ and $v$ are called \emph{marginals} or \emph{margins}. These polytopes are called transportation polytopes because they model the transportation of goods from $p$ supply locations (with the $i$th location supplying a quantity of $u_i$) to $q$ demand locations (with the $j$th location demanding a quantity of $v_j$). The feasible points $x = (x_{i,j})_{i \in [p], j \in [q]}$ in a $p \times q$ transportation polytope $P$ model the scenario where a quantity of $x_{i,j}$ of goods is transported from the $i$th supply location to the $j$th demand location. See Figure~\ref{figure:transportation-terminology}.
\begin{figure}[tbhp]
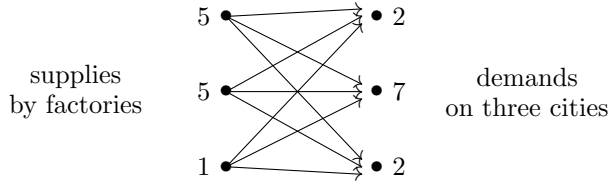

\begin{center}
\[
\xy
(0,0)*+{\bullet};
(0,-10)*+{\bullet}; 
(0,-20)*+{\bullet}; 
(20,0)*+{\bullet};
(20,-10)*+{\bullet}; 
(20,-20)*+{\bullet}; 
(-3,0)*+{5};
(-3,-10)*+{5}; 
(-3,-20)*+{1}; 
(23,0)*+{2};
(23,-10)*+{7}; 
(23,-20)*+{2}; 
{\ar (0,0)*{}; (18,1)*{}}; 
{\ar (0,0)*{}; (18,-9)*{}}; 
{\ar (0,0)*{}; (18,-19)*{}}; 
{\ar (0,-10)*{}; (18,0)*{}}; 
{\ar (0,-10)*{}; (18,-10)*{}}; 
{\ar (0,-10)*{}; (18,-20)*{}}; 
{\ar (0,-20)*{}; (18,-1)*{}}; 
{\ar (0,-20)*{}; (18,-11)*{}}; 
{\ar (0,-20)*{}; (18,-21)*{}}; 
(-20,-8)*+{\text{supplies}};
(-20,-12)*+{\text{by factories}};
(40,-8)*+{\text{demands}};
(40,-12)*+{\text{on three cities}};
\endxy
\]
\end{center}
\caption{Three supplies and three demands}\label{figure:transportation-terminology}
\end{figure}

\begin{example}\label{example:classical-tp-example}
Let us consider the $3 \times 3$ transportation polytope $P_{3 \times 3}$ defined by the marginals $u = (5,5,1)^T$ and $v = (2,7,2)^T$, which corresponds to the transportation problem shown in Figure~\ref{figure:transportation-terminology}. A point $x^* = (x^*_{i,j})$ in $P$ is shown in Figure~\ref{figure:classicalreindex}. The equations in~\eqref{equation:classicalsums} are conditions on the row sums and column sums (respectively) of tables $x \in P$.
\begin{figure}[htb]
\begin{equation*}
x^*=
\begin{tabular}{|c|c|c|}
\hline
$x^*_{1,1}$ & $x^*_{1,2}$ & $x^*_{1,3}$ \\ \hline
$x^*_{2,1}$ & $x^*_{2,2}$ & $x^*_{2,3}$ \\ \hline
$x^*_{3,1}$ & $x^*_{3,2}$ & $x^*_{3,3}$ \\ \hline
\end{tabular}
=
\begin{tabular}{|c|c|c|}
\hline
\ $2$\ &\ $2$\ &\ $1$\ \\ \hline
\ $0$\ &\ $5$\ &\ $0$\ \\ \hline
\ $0$\ &\ $0$\ &\ $1$\ \\ \hline
\end{tabular}
\end{equation*}
\caption{A point $x^* \in P_{3 \times 3}$}\label{figure:classicalreindex}
\end{figure}
\end{example}

\subsection{Dimension and feasibility}


Notice in Example~\ref{example:classical-tp-example} that $5+5+1=2+7+2$. The condition that the sum of the supply margins equals the sum of the demand margins is not only necessary but also sufficient for a classical transportation polytope to be non-empty:
\begin{lemma}\label{lemma:nonemptycriterion}
Let $P$ be the $p \times q$ classical transportation polytope defined by the marginals $u \in \R_{\geq 0}^p$ and $v \in \R_{\geq 0}^q$. The polytope $P$ is non-empty if and only if 
\begin{equation}\label{equation:classicalnonempty}
\sum_{i \in [p]} u_i = \sum_{j \in [q]} v_j.
\end{equation}
\end{lemma}
The proof of this lemma uses the well-known \emph{northwest corner rule algorithm} (see~\cite{Queyranne:MultiIndexTransportation} or Exercise 17 in Chapter 6 of~\cite{YKK}).


The equations \eqref{equation:classicalsums} and the inequalities $x_{i,j} \geq 0$ can be rewritten in the matrix form
\begin{equation*}
P = \{x \in \R^{pq} \mid Ax = b, x \geq 0\}
\end{equation*}
with a $0$-$1$ matrix $A$ of size $(p+q) \times pq$ and a vector $b \in \R^{p+q}$ called the \emph{constraint matrix}. The constraint matrix for a $p \times q$ transportation polytope is the vertex-edge incidence matrix of the complete bipartite graph $K_{p,q}$.
\begin{lemma}\label{lemma:classicalrankdimension}
Let $A$ be the constraint matrix of a $p \times q$ transportation polytope $P$. Then:
\begin{enumerate}
\item\label{classical-submatrices} Maximal rank submatrices of $A$ correspond to spanning trees on $K_{p,q}$.
\item\label{classical-rank} $\rank (A) = p+q-1$.
\item\label{classical-subdeterminant} Each subdeterminant of $A$ is $\pm 1$, thus $A$ is totally unimodular.
\item\label{classical-dimension} If $P \not= \emptyset$, its dimension is $pq-(p+q-1) = (p-1)(q-1)$.
\end{enumerate}
\end{lemma}
Part~\ref{classical-dimension} follows from Part~\ref{classical-rank}.

\begin{example}\label{example:classicalreindex}
Continuing from Example~\ref{example:classical-tp-example}, observe $P_{3 \times 3} = \{ x \in \R^9 \mid A_{3 \times 3}x = b, x \geq 0 \}$, where $A_{3 \times 3}$ is the constraint matrix
\begin{equation}
A_{3 \times 3} = \left[
\begin{array}{ccccccccc}
1&0&0&1&0&0&1&0&0 \\
0&1&0&0&1&0&0&1&0 \\
0&0&1&0&0&1&0&0&1 \\
1&1&1&0&0&0&0&0&0 \\
0&0&0&1&1&1&0&0&0 \\
0&0&0&0&0&0&1&1&1
\end{array}
\right]
\quad
\text{\rm and }
\quad
b = \left[
\begin{array}{c}
2 \\ 7 \\ 2 \\ 5 \\ 5 \\ 1
\end{array}
\right ]
.
\end{equation}
Up to permutation of rows and columns, the matrix $A_{3 \times 3}$ is the unique constraint matrix for $3 \times 3$ classical transportation polytopes. It is a $6 \times 9$ matrix of rank five. Thus, $P_{3 \times 3}$ is a four-dimensional polytope described in a nine-dimensional ambient space.
\end{example}

Birkhoff polytopes, first introduced by G. Birkhoff in~\cite{Birkhoff:TresObservaciones}, are an important subclass of transportation polytopes:
\begin{definition}\label{definition:BirkhoffPolytope}
The \emph{$p$th Birkhoff polytope}, denoted by $B_p$, is the $p \times p$ classical transportation polytope with margins $u=v=(1,1,\ldots,1)^T$.
\end{definition}
The Birkhoff polytope is also called the \emph{assignment polytope} or the \emph{polytope of doubly stochastic matrices} (see, e.g.,~\cite{BB}). It is the perfect matching polytope of the complete bipartite graph $K_{p,p}$. We can generalize the definition of the Birkhoff polytope to rectangular arrays:
\begin{definition}\label{definition:generalizedbirkhoffclassical}
The \emph{central transportation polytope} is the $p \times q$ classical transportation polytope with $u_1= \cdots = u_p = q$ and $v_1 = \cdots = v_q = p$.  This polytope is also called the \emph{generalized Birkhoff polytope} of size $p \times q$.
\end{definition}

\subsection{Combinatorics of faces and graphs}

The study of the faces of transportation polytopes is a nice combinatorial question (see, e.g.,~\cite{BR}). Unfortunately it is still incomplete, e.g., one does not
know the number of $i$-dimensional faces of each dimension other than in a few cases. E.g., in~\cite{Pak:Number}, Pak presented an efficient algorithm for computing the $f$-vector of the generalized Birkhoff polytope of size $p \times (p+1)$. Hartfiel (see~\cite{Hartfiel:Full-patterns}) and Dahl (see~\cite{Dahl:Transportation-matrices}) described the supports of certain feasible points in classical transportation polytopes.
In this section, we  fully describe the vertices and the edges of a $2$-way  transportation polytope $P$.  The resulting graph has some interesting
properties, but there are still open questions about it.

Let $P$ be a $p \times q$ classical transportation polytope. 
For a point $x = (x_{i,j})_{i \in [p], j \in [q]}$, define the \emph{support set} $\supp(x) = \{(i,j) \in [p] \times [q] \mid x_{i,j} > 0\}$.
We also define a bipartite graph $B(x)$, called the \emph{support graph} of $x$. The graph $B(x)$ is the following subgraph of the complete bipartite graph $K_{p,q}$: 
\begin{itemize}
\item {\bf Vertices of $B(x)$.} The vertices of the graph $B(x)$ are the vertices of the complete bipartite graph $K_{p,q}$. We label the supply nodes $\sigma_1,\dots,\sigma_p$ and the demand nodes $\delta_1,\dots,\delta_q$.
\item {\bf Edges of $B(x)$.} There is an edge $(\sigma_i,\delta_j)$ if and only if $x_{i,j}$ is \emph{strictly} positive. In other words, the edge set is indexed by $\supp(x)$.
\end{itemize}

\begin{example}
Let us consider the point $x^* \in P_{3 \times 3}$ from Example~\ref{example:classical-tp-example}. Here, $\supp(x^*)= \{(\sigma_1,\delta_1),(\sigma_1,\delta_2),(\sigma_1,\delta_3),(\sigma_2,\delta_2),(\sigma_3,\delta_3)\}$. Figure~\ref{figure:Bv1} depicts the graph $B(x^*)$.
\begin{figure}[hbt]
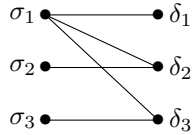

\[
\xy
(0,0)*+{\bullet};
(0,-7)*+{\bullet}; 
(0,-14)*+{\bullet}; 
(15,0)*+{\bullet};
(15,-7)*+{\bullet}; 
(15,-14)*+{\bullet}; 
(-3,0)*+{\sigma_1};
(-3,-7)*+{\sigma_2}; 
(-3,-14)*+{\sigma_3}; 
(18,0)*+{\delta_1};
(18,-7)*+{\delta_2}; 
(18,-14)*+{\delta_3}; 
(0,0); (15,0) **\dir{-}; 
(0,0); (15,-7) **\dir{-}; 
(0,0); (15,-14) **\dir{-}; 
(0,-7); (15,-7) **\dir{-}; 
(0,-14); (15,-14) **\dir{-}; 
\endxy
\]
\caption{The support graph $B(x^*)$ of the point $x^* \in P_{3 \times 3}$. The nodes of $B(x^*)$ on the left are the $p=3$ supplies. The nodes on the right are the $q=3$ demands.}\label{figure:Bv1}
\end{figure}
\end{example}


An important subclass of transportation polytopes are those which are generic. Generic transportation polytopes are easiest to analyze 
in the proofs which follow and are the ones typically appearing in applications. Generic $d$-way transportation polytopes are those whose vertices have maximal possible non-zero entries. All generic transportation polytopes are simple, but not vice versa.

\begin{definition}
A $p \times q$ classical transportation polytope $P$ is \emph{generic} if
\begin{equation}\label{equation:nondegenerateclassicalequation}
\sum_{i \in Y} u_i \not= \sum_{j \in Z} v_j.
\end{equation}
for every non-empty proper subset $Y \subsetneq [p]$ and non-empty proper subset $Z \subsetneq [q]$. (Of course, due to \eqref{equation:classicalnonempty}, we must disallow the case where $Y = [p]$ and $Z = [q]$.)
\end{definition}

The graph properties of $B(x)$ provide a useful combinatorial characterization of the vertices of classical transportation polytopes:
\begin{lemma}[Klee, Witzgall~\cite{Klee:FacesTransportation}]\label{lemma:vertexBx}
Let $P$ be a $p \times q$ classical transportation polytope defined by the marginals $u \in \R^p_{>0}$ and $v \in \R^q_{>0}$, and let $x \in P$. Then the graph $B(x)$ is spanning. The point $x$ is a vertex of $P$ if and only if $B(x)$ is a spanning forest. Moreover, if $P$ is generic, then $x$ is a vertex of $P$ if and only if $B(x)$ is a spanning tree.
\end{lemma}

\begin{corollary}\label{corollary:classicalTPsizeofsupport}
Let $x$ be a point in a generic $p \times q$ classical transportation polytope $P$. Then $x$ is a vertex of $P$ if and only if $|\supp(x)| = p+q-1$.
\end{corollary}

A vertex of a $p \times q$ transportation polytope is \emph{non-degenerate} if it has $p+q-1$ positive entries. Otherwise, the vertex is \emph{degenerate}. A transportation polytope is \emph{non-degenerate} if all its vertices are non-degenerate. 
Non-degenerate transportation polytopes are of particular interest, as they have the largest possible number of vertices and largest possible diameter among the graphs of all transportation polytopes of given type and parameters (e.g., $p$, $q$, and $s$). 
Indeed, if $P$ is a degenerate transportation polytope, by carefully perturbing the marginals that define $P$ we can get a non-degenerate polytope $P'$. (A careful explanation of how to do the perturbation is given in Lemma 4.6 of Chapter 6 in~\cite{YKK} on page 281.) The perturbed marginals are obtained by taking a feasible point $x$ in $P$, perturbing the entries in the table and using the recomputed sums as the new marginals for $P'$. The graph of $P$ can be obtained from that of $P'$ by contracting certain edges, which cannot increase either the diameter nor the number of vertices.

Finally, note the following property on the vertices of a classical transportation polytope, which follows from part~\ref{classical-subdeterminant} of Lemma~\ref{lemma:classicalrankdimension} and Cramer's rule:
\begin{corollary}
Given integral marginals $u,v$, all vertices of the corresponding transportation polytope are integral.
\end{corollary}

We now recall a classical characterization of the vertices of the Birkhoff polytope:
\begin{theorem}[Birkhoff-von Neumann Theorem]\label{theorem:birkhoffvonneumann}
The $p!$ vertices of the $p$th Birkhoff polytope $B_p$ are the $0$-$1$ permutation matrices of size $p \times p$.
\end{theorem}
In other words, the vertices of the Birkhoff polytope are the permutation matrices, so every doubly stochastic matrix is a convex combination of permutation matrices. This theorem was proved by Birkhoff in~\cite{Birkhoff:TresObservaciones} and proved independently by von Neumann (see~\cite{Neumann:A-certain-zero-sum}). Equivalent results were shown earlier in the thesis~\cite{Steinitz:Uber-die-Konstruction} of Steinitz, and the theorem also follows from~\cite{Konig:Grafok-es-alkalmazasuk} and~\cite{Konig:Uber-Graphen} by K{\H{o}}nig. For a more complete discussion, see the preface to~\cite{Lovasz:Matching}. See also the papers  \cite{Brualdi:DStochasticI, Brualdi:DStochasticII, Brualdi:DStochasticIII, Brualdi:DStochasticIV}, where various various combinatorial and geometric properties of the Birkhoff polytope were studied such as its graph. Of course due to the above theorem, Birkhoff's polytopes play an important role in combinatorics and discrete optimization and the literature about their properties is rather large. 

We also want to know how many vertices a transportation polytope can have. In
particular there is a visible difference in behavior between generic and
non-generic polytopes. How about maximum number of
vertices? The exact formula is complicated but the following result
of Bolker in~\cite{Bolker:Transportation-polytopes} can serve as a reference:
\begin{lemma}[Bolker, \cite{Bolker:Transportation-polytopes}]
The maximum possible number of vertices among $p \times q$ transportation polytopes is achieved by the central transportation polytope whose marginals are 
$u=(q,q,\dots,q)$ and $v=(p,p,\dots,p)$.
\end{lemma}
Indeed one can characterize which transportation polytopes reach the
largest possible number of vertices. (See results by Yemelichev, Kravtsov and collaborators from the 1970's mentioned in \cite{YKK}.)

\begin{question}
What are the possible values for the number of vertices of a generic $p \times q$ transportation polytope? Are there gaps or do all integer values on a interval occur?
\end{question}
A partial answer to this question is provided in Table~\ref{table:classical-f0}, with more detail available at~\cite{tpdb}. Another partial answer, given in~\cite{DeLoera:GraphsTP}, is:
\begin{theorem}\label{theorem:classical-gcd}
The number of vertices of a non-degenerate $p \times q$
classical transportation polytope is divisible by $\operatorname{gcd}(p,q)$.
\end{theorem}
\begin{table}[hbt]
\begin{center}
{\small
\begin{tabular}{|c|c|}
\hline
sizes & Distribution of number of vertices in transportation polytopes \cr \hline         
$2 \times 3$  &  3 4 5 6\cr \hline
$2 \times 4$  &  4 6 8 10 12 \cr \hline
$2 \times 5$  &  5 8 11 12 14 15 16 17 18 19 20 21 22 23 24 25 26 27 28 29 30 \cr \hline
$3 \times 3$  &  9 12 15 18 \cr \hline
$3  \times 4$ & 16 21 24 26 27 29 31 32 34 36 37 39 40 41 42 44 45 46 48 49 50 \cr 
              & 52 53 54 56 57 58 60 61 62 63 64 66 67 68 70 71 72 74 75 76 78 80 84 90 96 \cr \hline
$4  \times 4$ & 108 116 124 128 136 140 144 148 152 156 160 164 168 172 176 180 184 188 192 \cr
              &  196 200 204 208 212 216 220 224 228 232 236 240 244 248 252 256 260 264 268 \cr
              &  272 276 280 284 288 296 300 304 312 320 340 360 \cr \hline
\end{tabular}
}
\end{center}
\caption{Numbers of vertices of $p \times q$ transportation polytopes}\label{table:classical-f0}
\end{table}

The support graph associated to a point of the transportation polytope also characterizes edges of classical transportation polytopes. (See Lemma 4.1 in Chapter 6 of~\cite{YKK}.)
\begin{proposition}
Let $x$ and $x'$ be distinct vertices of a classical transportation polytope $P$. Then the vertices $x$ and $x'$ are adjacent if and only if the graph $B(x) \cup B(x')$ contains a unique cycle.
\end{proposition}
This can be seen since the bases corresponding to the vertices $x$ and $x'$ differ in the addition and the removal of one element (see~\cite{Matousek:LinearProgramming, Schrijver:LinearProgramming}). 

One can also characterize the facets of the $p \times q$ transportation polytope, which have dimension $(p-1)(q-1)-1$ by Lemma~\ref{lemma:classicalrankdimension}. The following lemma is Theorem 3.1 in Chapter 6 of~\cite{YKK}.
\begin{lemma}\label{lemma:classical-facet-characterization}
Let $P$ be the $p \times q$ transportation polytope ($pq > 4$) defined by marginals $u$ and $v$.
Pick integers $1 \leq i^* \leq p$ and $1 \leq j^* \leq q$.
The subset of points of $P$
\[F_{i^*,j^*} = \{ (x_{i,j}) \in P \mid x_{i^*,j^*} = 0\}\]
is a facet of $P$ if and only if $u_{i^*}+v_{j^*} < \sum_{i=1}^p u_i$.
\end{lemma}
\begin{figure}[hbt]
\begin{center}
\[
\xy
(0,0); (15,0) **\dir{-}; 
(0,-5); (15,-5) **\dir{-}; 
(0,-10); (15,-10) **\dir{-}; 
(0,-15); (15,-15) **\dir{-}; 
(0,0); (0,-15) **\dir{-}; 
(5,0); (5,-15) **\dir{-}; 
(10,0); (10,-15) **\dir{-}; 
(15,0); (15,-15) **\dir{-}; 
(12.5,-12.5)*+{\checkmark}; 
(2.5,-2.5)*+{\times}; 
(19,-2.5)*+{100}; 
(18,-7.5)*+{6}; 
(18,-12.5)*+{6}; 
(2.5,-18)*+{38}; 
(7.5,-18)*+{37}; 
(12.5,-18)*+{37}; 
\endxy
\]
\end{center}
\caption{The equation $x_{3,3}=0$ defines a facet, while the top-left corner entry corresponds to an equation $x_{1,1}=0$ that does not.}\label{figure:facetsvsnonfacet}
\end{figure}
See Figure~\ref{figure:facetsvsnonfacet} for an example. From this basic characterization we see:
\begin{corollary}\label{corollary:classicalnumberfacets}
For $2 \leq p \leq q$ and $q\geq 3$, the possible 
number of facets of a $p \times q$ transportation polytope is a number
of the form $(p-1)q+k$ for $k=0,\dots,q$ and only such integers can occur.
\end{corollary}
For example, $3 \times 3$ transportation polytopes can have $6$, $7$, $8$, or $9$ facets and only these values occur. 

\subsubsection{Diameter of graphs of transportation polytopes}

Now we study a classical question about the graphs of transportation polytopes. Recall that the \emph{distance} between two vertices $x,y$ of a polytope $P$ is the minimal number $\operatorname{dist}_P(x,y)$ of edges needed to go from $x$ to $y$ in the graph of $P$. The \emph{diameter} of a polytope is the maximum possible distance between pairs of vertices in the graph of the polytope. Though the Hirsch Conjecture was finally shown to be  false in general for polytopes  (see~\cite{Santos:CounterexampleHirsch}), the problem is still unsolved for transportation polytopes, and diameter bounds 
for this special class of polytopes are very interesting. Dyer and Frieze (see~\cite{Dyer:RandomWalks}) gave the first polynomial diameter bound for totally unimodular polytopes which applies to classical transportation polytopes (and more generally to network polytopes), but this was recently improved by Bonifas et al. in~\cite{bonifasetal}.

The diameters of classical transportation polytopes and their applications (see, e.g.,~\cite{CDMS}) have been studied extensively. In~\cite{Balinski:DualTransportation}, Balinski proved that the Hirsch Conjecture holds and is tight for dual transportation polyhedra.  For the specific case of transportation polytopes Yemelichev, Kovalev, and Kravtsov (see Theorem 4.6 in Chapter 6 of~\cite{YKK} and the references therein) and Stougie (see~\cite{Stougie:PolynomialBound}) presented improved polynomial bounds. This was improved to a quadratic bound by van~den~Heuvel and Stougie in~\cite{vandenHeuvel:QuadraticBound}, and further improved to a linear bound: 

\begin{theorem}[Brightwell, van den Heuvel, Stougie~\cite{Brightwell:LinearTransportation}] \label{stougieetal}
The diameter of every $p \times q$ transportation polytope is at most $8(p+q-2)$.
\end{theorem}
The bound follows from a crucial lemma which bounds the graph distance $\operatorname{dist}_P(y,y')$ between any two vertices $x$ and $y$ of a $p \times q$ transportation polytope $P$, by constructing vertices $x'$ and $y'$ of $P$ and nodes $\sigma, \delta$ of $K_{p,q}$ such that $\deg_{B(x')}(\delta)=\deg_{B(y')}(\delta)=1$, $(\sigma,\delta) \in B(x') \cap B(y')$, and $\dist_P(x,x')+\dist_P(y,y') \leq 8$. In the arguments below, there is an important distinction between {\bf vertices} of the polytope $P$ (which we always denote by $x$ or $y$) and {\bf nodes} of the support graph $B(x) \subset K_{p,q}$ of a vertex $x$ of $P$ (which we always denote by $\sigma$ or $\delta$).

Theorem \ref{stougieetal}  was further improved by Cor Hurkens~\cite{Hurkens:Diameter4p}.

\begin{theorem}[Hurkens~\cite{Hurkens:Diameter4p}]
The diameter of every $p \times q$ transportation polytope is at most $4(p+q-2)$.
\end{theorem}

We present a brief sketch of Hurkens' proof. The result follows immediately from this lemma:
\begin{lemma}[Hurkens~\cite{Hurkens:Diameter4p}]\label{lemma:crucial-hurkens}
For any two vertices $x$ and $y$ of a $p \times q$ transportation polytope $P$, there is an integer $r \geq 1$, a vertex $y'$ of $P$, and nodes $\sigma, \delta_1, \ldots, \delta_r$ of $K_{p,q}$ such that:
\begin{enumerate}
	\item $\deg_{B(x)}(\delta_k)=\deg_{B(y')}(\delta_k)=1$ for $k=1,\dots,r$,
	\item $(\sigma,\delta_k) \in B(x), B(y')$ for $k=1,\dots,r$, and 
	\item $\dist_P(y,y') \leq 4r$.
\end{enumerate}
\end{lemma}
The key idea that Hurkens showed is that four pivots are required (on average) to construct a common leaf node. More specifically, Hurkens proved this lemma by showing that for any two vertices $x$ and $y$ of a transportation polytope $P$, there is a node $\sigma$ in $K_{p,q}$ (which can be assumed to be a supply) with $r$ incident edges $(\sigma,\delta_1),\dots,(\sigma,\delta_r)$ in $B(x)$ where $\delta_1,\dots,\delta_r$ are all leaf nodes (which are necessarily demands) of $K_{p,q}$. Moreover, the nodes $\sigma,\delta_1,\dots,\delta_r$ of $K_{p,q}$ identified in Hurkens' algorithm also satisfy the property that if 
\[S := \{(\sigma,\delta_k) \mid (\sigma,\delta_k) \in B(y), \,\, k=1,\dots,r\},\]
then there is a vertex $y'$ of $P$ obtained after at most $4r$ pivots from the vertex $y$ of $P$ such that $B(x)$ and $B(y')$ have $r$ common leaf nodes.

In the algorithm of Brightwell, van den Heuvel, and Stougie (see~\cite{Brightwell:LinearTransportation}), pivots are applied to vertices $x$ and $y$ of $P$, resulting in new vertices $x'$ and $y'$ of $P$. A key difference in Hurkens' algorithm in~\cite{Hurkens:Diameter4p} is that pivots are only applied to \emph{one} of the two vertices $x$ and $y$ of $P$. Without loss of generality, pivots are applied to the vertex $y$ of $P$ and not applied to the vertex $x$ of $P$. Thus, we do not describe the vertex $x$ further. Other than the property  that the demand nodes $\delta_1,\dots,\delta_r$ are leaf nodes in $B(x)$ adjacent to the node $\sigma$, the structure of $B(x)$ may be arbitrary.

We label the relevant supply and demands nodes participating in pivots. For each $k=1,\dots,r$ let $(\sigma_{k,n},\delta_k)$ for $n=1,\dots, \ell_k$ be the edges in $B(y) \setminus S$ incident to $\delta_k$, where $\ell_k = \deg_{B(y)\setminus S}(\delta_k)$. 
Let $(\sigma,c_q)$ be the edges in $B(y) \setminus S$ incident to $\sigma$ for $q=1,\dots,t$ where $t=\deg_{B(y)\setminus S}(\sigma)$. See Figure~\ref{figure:hurkens-lemma-2}.
\begin{figure}[hbt]
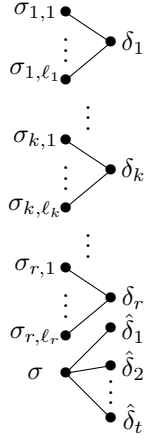

\begin{center}
\[ 
\xy
(4,0)*+{\bullet}; %
(4,-4)*+{\vdots}; %
(4,-9)*+{\bullet}; %
(0,0)*+{\sigma_{1,1}};
(0,-4)*+{};
(0,-8)*+{\sigma_{1,\ell_1}};
(10,-4)*+{\bullet}; 
(13,-4)*+{\delta_1};
(4,0); (10,-4) **\dir{-};
(4,-9); (10,-4) **\dir{-};
(7,-13)*+{\vdots}; %
(4,-17)*+{\bullet}; %
(4,-21)*+{\vdots}; %
(4,-26)*+{\bullet}; %
(0,-17)*+{\sigma_{k,1}};
(0,-21)*+{};
(0,-26)*+{\sigma_{k,\ell_k}};
(10,-21)*+{\bullet}; 
(13,-21)*+{\delta_k};
(4,-17); (10,-21) **\dir{-};
(4,-26); (10,-21) **\dir{-};
(7,-30)*+{\vdots}; %
(4,-34)*+{\bullet}; %
(4,-38)*+{\vdots}; %
(4,-43)*+{\bullet}; %
(0,-34)*+{\sigma_{r,1}};
(0,-38)*+{};
(0,-43)*+{\sigma_{r,\ell_r}};
(10,-38)*+{\bullet}; 
(13,-38)*+{\delta_r};
(4,-34); (10,-38) **\dir{-};
(4,-43); (10,-38) **\dir{-};
(4,-48)*+{\bullet}; %
(0,-48)*+{\sigma};
(10,-42)*+{\bullet}; %
(10,-47)*+{\bullet}; %
(10,-49.5)*+{\vdots}; %
(10,-54)*+{\bullet}; %
(13,-42)*+{\hat\delta_1}; %
(13,-47)*+{\hat\delta_2}; %
(13,-54)*+{\hat\delta_t}; %
(4,-48); (10,-42) **\dir{-};
(4,-48); (10,-47) **\dir{-};
(4,-48); (10,-54) **\dir{-};
\endxy
\]
\end{center}
\caption{Example of supply nodes adjacent to demand nodes $\delta_k$ in $B(y)$ for $k=1,\dots,r$ and demand nodes adjacent to the supply node $\sigma$}\label{figure:hurkens-lemma-2}
\end{figure}
Here we describe the successive pivots applied starting from the vertex $y$ of $P$. For each $k=1,\dots,r$, we do the following:
\begin{enumerate}
\item If $(\sigma,\delta_1)$ is not in the support graph, pivot to add $(\sigma,\delta_1)$. Then, pivot to add edges of the form $(\sigma_{1,n},\hat\delta_q)$ for $n=1,2,\dots$ until all edges of the form $(\sigma_{1,n},\delta_1)$ are removed.

\item If $(\sigma,\delta_2)$ is not in the support graph, pivot to add $(\sigma,\delta_2)$. Then, pivot to add edges of the form $(\sigma_{2,n},\hat\delta_q)$ for $n=1,2,\dots$  until all edges of the form $(\sigma_{2,n},\delta_2)$ are removed.

\item Continue in this way for $k=3,\dots,r$: If $(\sigma,\delta_k)$ is not in the support graph, pivot to add it. Then, pivot to add edges of the form $(\sigma_{k,n},\hat\delta_q)$ for $n=1,2,\dots$ until all edges of the form $(\sigma_{k,n},\delta_k)$ are removed.

\end{enumerate}
In the resulting vertex $y'$ of $P$, the support graph $B(y')$ has $\delta_1,\dots,\delta_r$ as leaf nodes adjacent to $\sigma$, which matches the support graph $B(x)$ of the vertex $x$ of $P$. What remains to show (and we skip it) is that there is a choice of nodes $\sigma,\delta_1,\dots,\delta_r$ where the number of pivots performed is at most $4r$. Instead, we illustrate the idea behind the sequence of prescribed pivots in an example:

\begin{example}
Let $y$ be a vertex of $P$ where nodes $\sigma,\delta_1,\dots,\delta_r$ in $B(y)$ are already identified. Figure~\ref{figure:hurkens-lemma-2-example} shows the support graph $B(y)$. (The vertex $x$ and its associated support graph $B(x)$ can be arbitrary, thus we do not depict it in Figure~\ref{figure:hurkens-lemma-2-example}.)

Since $(\sigma,\delta_1)$ is not in the support graph $B(y)$ of the vertex $y$ of $P$, we insert it, and the pivot operation removes the edge $(\sigma_{1,3},\delta_1)$. We now apply pivots to the resulting adjacent vertex of $P$ as follows: After the pivot, only the edges $(\sigma_{1,1},\delta_1)$ and $(\sigma_{1,2},\delta_1)$ are incident to the demand node $\delta_1$. These two edges are removed by pivoting to add the edges $(\sigma_{1,1},\hat\delta_1)$ and $(\sigma_{1,2},\hat\delta_1)$, respectively, which causes $\delta_1$ to be a leaf node adjacent to $\sigma$.

After insertion of the edge $(\sigma,\delta_2)$ the remaining edge of the form  $(\sigma_{2,n},\delta_2)$ is removed the same way. Since $\delta_3$ is already a leaf node, the insertion of $(\sigma,\delta_3)$ will cause it to be a leaf node adjacent to $\sigma$.
\begin{figure}[hbt]
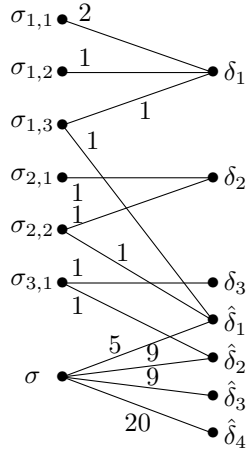

\begin{center}
\[
\xy
(0,0)*{}="s11"; 
(0,-7)*{}="s12"; 
(0,-14)*{}="s13"; 
"s11"*+{\bullet}; 
"s11"+(-4,0)*+{\sigma_{1,1}};
"s12"*+{\bullet}; 
"s12"+(-4,0)*+{\sigma_{1,2}};
"s13"*+{\bullet}; 
"s13"+(-4,0)*+{\sigma_{1,3}};
(20,-7)*{}="d1"; 
"d1"*+{\bullet}; 
"d1"+(3,0)*+{\delta_1};
"s11"; "d1" **\dir{-};
"s12"; "d1" **\dir{-};
"s13"; "d1" **\dir{-};
"s11"+(3,1)*+{2}; 
"s12"+(3,2)*+{1}; 
"s13"+(11,2)*+{1}; 
"s13"+(4,-2)*+{1}; 
(0,-21)*{}="s21"; 
(0,-28)*{}="s22"; 
"s21"*+{\bullet}; 
"s21"+(-4,0)*+{\sigma_{2,1}};
"s22"*+{\bullet}; 
"s22"+(-4,0)*+{\sigma_{2,2}};
(20,-21)*{}="d2"; 
"d2"*+{\bullet}; 
"d2"+(3,0)*+{\delta_2};
"s21"; "d2" **\dir{-};
"s22"; "d2" **\dir{-};
"s21"+(2,-2)*+{1}; 
"s22"+(2,2.1)*+{1}; 
"s22"+(8,-3)*+{1}; 
(0,-35)*{}="s31"; 
"s31"*+{\bullet}; 
"s31"+(-4,0)*+{\sigma_{3,1}};
(20,-35)*{}="d3"; 
"d3"*+{\bullet}; 
"d3"+(3,0)*+{\delta_3};
"s31"; "d3" **\dir{-};
"s31"+(2,2.1)*+{1}; 
"s31"+(2,-3)*+{1}; 
(20,-40)*{}="dh1"; 
(20,-45)*{}="dh2"; 
(20,-50)*{}="dh3"; 
(20,-55)*{}="dh4"; 
"dh1"*+{\bullet};
"dh2"*+{\bullet};
"dh3"*+{\bullet};
"dh4"*+{\bullet};
"dh1"+(3,0)*+{\hat\delta_1};
"dh2"+(3,0)*+{\hat\delta_2};
"dh3"+(3,0)*+{\hat\delta_3};
"dh4"+(3,0)*+{\hat\delta_4};
(0,-47.5)*{}="s"; 
"s"*+{\bullet}; 
"s"+(-4,0)*+{\sigma};
"s"; "dh1" **\dir{-};
"s"; "dh2" **\dir{-};
"s"; "dh3" **\dir{-};
"s"; "dh4" **\dir{-};
"s13"; "dh1" **\dir{-};
"s22"; "dh1" **\dir{-};
"s31"; "dh2" **\dir{-};
"s"+(7,4.3)*+{5}; 
"s"+(12,3.2)*+{9}; 
"s"+(12,0)*+{9}; 
"s"+(10,-6)*+{20}; 
\endxy
\]
\end{center}
\caption{The support graph $B(y)$  of the vertex $y$ of a transportation polytope $P$}\label{figure:hurkens-lemma-2-example}
\end{figure}
\end{example}

To prove that the Hirsch Conjecture is true for transportation polytopes, one would hope that any pair of vertices that differ in $k$ support elements has a pivot step that reduces the number of non-zero variables in which the vertices differ, but Brightwell et al.{}~\cite{Brightwell:LinearTransportation} noticed that this was not true. We show their counter-example in Figure~\ref{figure:reducingpivot}.
\begin{figure}[hbt]
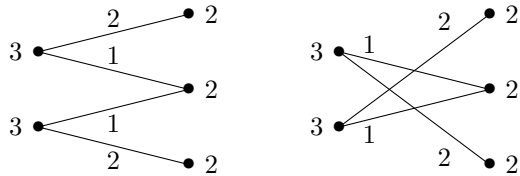

\begin{center}
\[
\xy
(0,-5)*+{\bullet}; 
(0,-15)*+{\bullet}; 
(-3,-5)*+{3}; 
(-3,-15)*+{3}; 
(20,0)*+{\bullet};
(20,-10)*+{\bullet}; 
(20,-20)*+{\bullet}; 
(23,0)*+{2}; 
(23,-10)*+{2}; 
(23,-20)*+{2}; 
(0,-5); (20,0) **\dir{-}; %
(0,-5); (20,-10) **\dir{-}; %
(0,-15); (20,-10) **\dir{-}; %
(0,-15); (20,-20) **\dir{-}; %
(10,-.5)*+{2}; 
(10,-5.5)*+{1}; 
(10,-14.5)*+{1}; 
(10,-19.5)*+{2}; 
(40,-5)*+{\bullet}; 
(40,-15)*+{\bullet}; 
(37,-5)*+{3}; 
(37,-15)*+{3}; 
(60,0)*+{\bullet};
(60,-10)*+{\bullet}; 
(60,-20)*+{\bullet}; 
(63,0)*+{2}; 
(63,-10)*+{2}; 
(63,-20)*+{2}; 
(40,-5); (60,-20) **\dir{-}; %
(40,-5); (60,-10) **\dir{-}; %
(40,-15); (60,-10) **\dir{-}; %
(40,-15); (60,0) **\dir{-}; %
(54,-1)*+{2}; 
(54,-19)*+{2}; 
(44,-4)*+{1}; 
(44,-16)*+{1}; 
\endxy
\]
\end{center}
\caption{Support graphs of a pair of vertices where no pivot reduces the difference in support}\label{figure:reducingpivot}
\end{figure}

\begin{openproblem}
Prove or disprove the Hirsch Conjecture for $2$-way transportation polytopes.
\end{openproblem}

By Corollary~\ref{corollary:classicalnumberfacets}, this would mean the diameter is less than or equal to $p+q-1$. 
This conjecture holds for many special cases that restrict the margins. For example the conjecture is true
for Birkhoff's polytope and for some special right-hand sides (see e.g., \cite{borgwardt}).

While transportation polytopes seem tame compared to other polytopes. It has been shown that they have some non-trivial topological structure: Diameter bounds for simple $d$-polyhedra can be studied via decomposition properties of related simplicial complexes. Each non-degenerate
simple polytope has a polar \emph{simplicial complex}, a simplicial polytope.  Billera and Provan (see~\cite{provanbillera}) showed that polytopes whose dual simplicial polytope is \emph{weakly vertex decomposable} have a linear diameter. But it has recently been shown (see~\cite{DeLoeraKlee:NotAllSimplicial}) that the infinite family of polars of $p \times 2$ transportation polytopes for $p \geq 5$ are not weakly vertex-decomposable, the first ever such examples. But at the same time, one can prove the Hirsch Conjecture holds for $p \times 2$ transportation polytopes by  proving a stronger statement:

\begin{theorem}
The Hirsch Conjecture holds for all convex polytopes obtained as the intersection of a cube and a hyperplane.
\end{theorem}

Fix a dimension $d \in \N$.  Let $H = \{x \in \R^d \mid a_1x_1 + \cdots + a_dx_d = b\}$ be the hyperplane determined by the non-zero normal vector $a = (a_1, \ldots, a_d)$ and t
he constant $b \in \R$.  Let $\Box_d$ denote $d$-dimensional cube with $0$-$1$ vertices. Then, let $P$ denote the polytope obtained as their intersection $P = \Box_d \cap H$.

If the dimension of the polytope $P$ is less than $d-1$, then $P$ is a face of $\Box_d$. In that case, $P$ itself is a cube of lower dimension, so we assume that the polytope $P
$ is of dimension $d-1$. We may also assume that the polytope $P$ is not a facet of the $d$-cube, so that $H$ intersects the relative interior of $\Box_d$.

Without assuming any genericity, a simple dimension argument shows that the vertices of the polytope $P$ are either on the relative interior of an edge of the cube $\Box_d$ or are vertices of the cube. We assume that $H$ is sufficiently generic.  Then, no vertex of the cube $\Box_d$ will be a vertex of $P$. For each vertex $v$ of $P$, we define its \emph{side signature} $\sigma(v)$ to be a string of length $d$ consisting of the characters $*$, $0$, and $1$ by the following rule:
\begin{equation}
\sigma(v)_i = \left\{
\begin{array}{ll}
0 & \hbox{if } v_i = 0,\\
1 & \hbox{if } v_i = 1,\\
* & \hbox{if } 0 < v_i < 1.\\
\end{array}
\right.
\end{equation}
By genericity, it cannot be the case that there are two vertices of $P$ with the same side signature.  Indeed, if there were two distinct vertices $v$ and $w$ with the same side
 signature, then $P$ will contain the entire edge of the cube containing them both, and $v$ and $w$ will not be vertices.

Let $H_{i,0}$ denote the hyperplane $\{x \in \R^d \mid x_i = 0\}$ and let $H_{i,1}$ denote the hyperplane $\{x \in \R^d \mid x_i = 1\}$. If there is an $i \in [d]$ such that the
 hyperplane $H$ does not intersect $H_{i,0}$ nor $H_{i,1}$, then we can project $P$ to a lower-dimensional face of $I_d$. Thus, for each $i \in [d]$, we can assume that $H$ inte
rsects at least one of $H_{i,0}$ or $H_{i,1}$.

Given two vertices $v=(v_1,\ldots,v_d)$ and $w=(w_1,\ldots,w_d)$ of $P = I_d \cap H$, we define the \emph{Hamming distance} between them based on their side signatures: 
\begin{equation}
\hamm(v,w) = \sum_{i=1}^d \hamm( \sigma(v)_i, \sigma(w)_i ),
\end{equation}
where
\begin{equation*}
\hamm(0,1) = \hamm(1,0) = \hamm(1,*) = \hamm(*,1) = \hamm(0,*) = \hamm(*,0) = 1
\end{equation*}
and
\begin{equation*}
\hamm(0,0)=\hamm(1,1)=\hamm(*,*) = 0.
\end{equation*}

\begin{lemma}
Let $P$ defined as above using a sufficiently-generic hyperplane $H$.  Let $v$ and $w$ be two vertices of $P$.  Let $f(P)$ denote the number of facets of $P$.

If $v$ and $w$ have the $*$ in the same coordinate and $P$ does not intersect either of the two facets in that direction, i.e., there is an $i$ such that $\sigma(v)_i=*=\sigma(w)_i$ and $P \cap H_{i,1} = \emptyset = P \cap H_{i,0}$, then $f(P) \geq (d-1) + \hamm(v,w)-1$.

Otherwise, $f(P) \geq (d-1) + \hamm(v,w)$.
\end{lemma}
\begin{proof}
By rotating the (combinatorial) cube if necessary, we can assume without loss of generality that the side signature $\sigma(v)$ of the vertex $v$ is $(*,0,0,\ldots,0)$ and that the side signature $\sigma(w)$ of the vertex $w$ is either of the form $(0,*,0,0,\ldots,0,1,1,\ldots,1)$ with $k \geq 0$ trailing ones or of the form $(*,0,0,\ldots,0,1,1,\ldots,1)$ with $k \geq 1$ trailing ones, after applying a suitable rotation to the cube.

In the first case, $\hamm(v,w)=k+1$ and we have at least $d$ ``$0$-facets'' and $k$ ``$1$-facets.''

In the second case, we have $d-1$ ``$0$-facets'', $k$ ``$1$-facets'' and (unless there is an $i$ such that $\sigma(v)_i=*=\sigma(w)_i$ and $P \cap H_{i,1} = \emptyset = P \cap H_{i,0}$) at least one more facet.  Thus, $f(P) \geq (d-1) + k + 1 = (d-1) + \hamm(v,w)$, unless we are in the special case, in which case $f(P) \geq (d-1) + k + 1 - 1$.
\end{proof}

\begin{lemma}
Let $P$ defined as above using a sufficiently-generic hyperplane $H$.  Let $v$ and $w$ be two vertices of $P$.  Then, there is a pivot from the vertex $v$ to a vertex $v'$ with $\hamm(v',w) = \hamm(v,w)-1$.
\end{lemma}
\begin{proof}
Again by rotating if necessary, without loss of generality, we can assume that $\sigma(v) = (*,0,0,\ldots,0)$ and that $\sigma(w)$ is either $(*,1,1,\ldots,1)$ or $(1,1,\ldots,1,*)$.

If the side signature $\sigma(w)$ of $w$ is $(*,1,1,\ldots,1)$, performing a pivot on the vertex $v$ in any one of the $d-1$ last coordinates reduces the Hamming distance.

Otherwise, the side signature $\sigma(w)$ of $w$ is $(1,1,\ldots,1,*)$. We now describe what can occur when pivoting from the vertex $v$ to a new vertex $v'$. We claim that at least one of the $d-1$ possible pivots on the vertex $v$ does not put a $0$ in the first coordinate of the side signature $\sigma(v')$ of the new vertex $v'$. Otherwise, the hyperplane $H$ cuts the polytope $P$ as a vertex figure: that is to say, the polytope $P$ cuts the corner $(1,0,\ldots,0)$ of the cube. See Figure~\ref{figure:cubehyperplanevertexfigure} for a picture.
\begin{figure}[hbt]
  \begin{center}
    \includegraphics[scale=0.9]{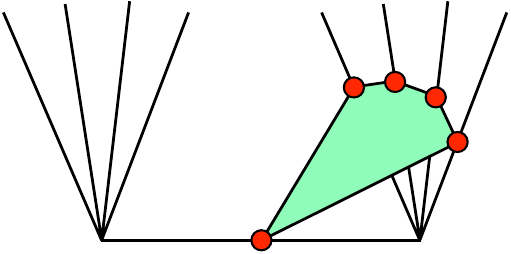}
    \caption{The vertex $v$ on the horizontal axis (the $x_1$ coordinate increases moving to the right) and its neighboring vertices on orthogonal edges of the cube.} \label{figure:cubehyperplanevertexfigure}
  \end{center}
\end{figure}

The remaining kind of pivots on $v$ that result in a new vertex $v'$ give side signatures $\sigma(v')$ of one of the following three forms:
\begin{enumerate}
\item The signature $\sigma(v')$ of the neighbor $v'$ of the vertex $v$ could be \[(1,0,0,\ldots,*,0,\ldots,0),\] which reduces the Hamming distance by one.
\item The signature $\sigma(v')$ of the neighbor $v'$ of the vertex $v$ could be \[(*,0,0,\ldots,0,1,0,0,\ldots,0),\] which reduces the Hamming distance by one.
\item Otherwise, \emph{one} remaining pivot could give the side signature \[(1,0,0,\ldots,0,*)\] for $\sigma(v')$.
\end{enumerate}
This third type of pivot does not reduce the Hamming distance. But if this is the only pivot that could give this and the first two kinds of pivots cannot be performed, then all of the remaining pivots are the kind that put $0$ in the first coordinate of the side signature $\sigma(v)$ of $v$. But this would imply that $H$ could not have intersected the hyperplane $H_{d+}$, and thus $P$ would be a $(d-1)$-cube with one vertex truncated.
\end{proof}

\begin{corollary}\label{theorem:thm_2n}
Let $P \not= \emptyset$ be a classical transportation polytope of size $p \times 2$ with $n \leq 2p$ facets. Then, the dimension of $P$ is $d=p-2$ and the diameter of $P$ is at  most $n - d$.
\end{corollary}

To see this follows from the previous theorem, we note that the coordinate-erasing projection of $P$ to the coordinates $x_{1,1}, x_{2,1}, \ldots, x_{p,1}$ of the first column shows that $P$ is the intersection of a hyperplane with a rectangular prism. (In particular, if the intervals are all equal and one has a cube, then the Minkowski sum of two consecutive hypersimplices $D(p,i)$ and $D(p,i+1)$ can be realized as a transportation polytope of size $p \times 2$.) After an affine transformation, the polytope $P$ is the intersection of a hyperplane and a cube. (The transformation takes the cube $[0,u_1] \times \cdots \times [0,u_p]$ to the cube $[0,1]^p$. That is to say, the $i$th coordinate $y_i$ in the cube $[0,1]^p$ is $
x_{i,1}/u_i$.)  By applying an affine transformation to $P = \Box_d \cap H$, we obtain a $p \times 2$ classical transportation polytope.

The Hirsch bound also holds for Birkhoff polytopes:

\begin{theorem}
Let $B_p$ be the $p$th Birkhoff polytope then
\begin{enumerate}

\item\label{item:birkhoff-vertex-degree} the degree of each vertex of $B_p$ is 
\[\sum_{k=0}^{p-2} \binom{p}{k}(p-k-1)!\]

\item\label{item:birkhoff-diameter} If $p \geq 4$, the diameter of $B_p$ is $2$.

\item (Billera-Sarangarajan~\cite{Billera:CombinatoricsPermutationPolytopes}) Every pair of vertices $x,y$ is contained in
a cubical face. The dimension of this cubical face is the number of cycles in the union of $B(x)$ and $B(y)$.
\end{enumerate}
\end{theorem}
\begin{proof}

For part~\ref{item:birkhoff-vertex-degree}, note that because the symmetric group 
acts transitively on the vertices (which are permutation matrices) the degree of all 
vertices is the same. It suffices to count how many vertices are adjacent to the vertex 
corresponding to the identity matrix. Any adjacent vertex $y$ has $k$ common edges with $x$ for $k=0,\dots,p-2$.
Now the $k$ edges can be chosen in $\binom{p}{k}$ ways and for each choice we
have a unique cycle being formed with the remaining $(p-k-1)$ pairs of vertices $(i,i')$. This can be done in $(p-k-1)!$ ways.

Now we prove part~\ref{item:birkhoff-diameter}. Given two non-adjacent vertices $x$ 
and $y$ we have a third vertex $z$ adjacent to both. Without loss of generality, the two graphs
$B(x)$ and $B(y)$ have no common edges, otherwise apply induction. Thus they define $p$ 
disjoint bipartite cycles, as shown in Figure~\ref{figure:diameter-birkhoff}.
\end{proof}
\begin{figure}[hbt]
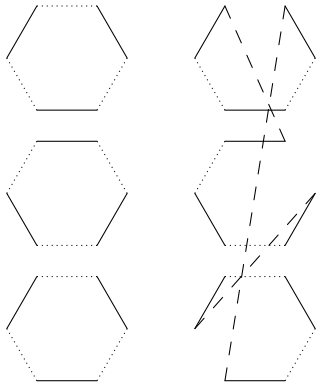

\begin{center}
\[ 
\xy 
(8,0)*{}="A"; 
(4,6.9282)*{}="B";
(-4,6.9282)*{}="C";
(-8,0)*{}="D"; 
(-4,-6.9282)*{}="E";
(4,-6.9282)*{}="F";
"A"; "B" **\dir{-};
"B"; "C" **\dir{.};
"C"; "D" **\dir{-};
"D"; "E" **\dir{.};
"E"; "F" **\dir{-};
"F"; "A" **\dir{.};
(0,-18)*{}="L2";
"A"+"L2"; "B"+"L2" **\dir{.};
"B"+"L2"; "C"+"L2" **\dir{-};
"C"+"L2"; "D"+"L2" **\dir{.};
"D"+"L2"; "E"+"L2" **\dir{-};
"E"+"L2"; "F"+"L2" **\dir{.};
"F"+"L2"; "A"+"L2" **\dir{-};
(0,-36)*{}="L3";
"A"+"L3"; "B"+"L3" **\dir{-};
"B"+"L3"; "C"+"L3" **\dir{.};
"C"+"L3"; "D"+"L3" **\dir{-};
"D"+"L3"; "E"+"L3" **\dir{.};
"E"+"L3"; "F"+"L3" **\dir{-};
"F"+"L3"; "A"+"L3" **\dir{.};
(25,0)*{}="R1";
"A"+"R1"; "B"+"R1" **\dir{-};
"C"+"R1"; "D"+"R1" **\dir{-};
"D"+"R1"; "E"+"R1" **\dir{.};
"E"+"R1"; "F"+"R1" **\dir{-};
"F"+"R1"; "A"+"R1" **\dir{.};
(25,-18)*{}="R2";
"B"+"R2"; "C"+"R2" **\dir{-};
"C"+"R2"; "D"+"R2" **\dir{.};
"D"+"R2"; "E"+"R2" **\dir{-};
"E"+"R2"; "F"+"R2" **\dir{.};
"F"+"R2"; "A"+"R2" **\dir{-};
(25,-36)*{}="R3";
"A"+"R3"; "B"+"R3" **\dir{-};
"B"+"R3"; "C"+"R3" **\dir{.};
"C"+"R3"; "D"+"R3" **\dir{-};
"E"+"R3"; "F"+"R3" **\dir{-};
"F"+"R3"; "A"+"R3" **\dir{.};
"C"+"R1"; "B"+"R2" **\dir{--};
"B"+"R1"; "E"+"R3" **\dir{--};
"A"+"R2"; "D"+"R3" **\dir{--};
\endxy 
\]
\end{center}
\caption{Diameter of Birkhoff polytope}\label{figure:diameter-birkhoff}
\end{figure}

It is worth noting that even if the Hirsch Conjecture for transportation polytopes is true, 
the simplex method may behave badly because there could be long decreasing pivot sequences:
\begin{theorem}[I. Pak~\cite{Pak:FourQuestionsBirkhoff}] Consider the linear functional
\[c_\alpha \cdot x=x_{1,1}+\alpha x_{1,2}+\dots+ \alpha^{p-1} x_{1,p}+\alpha^p x_{2,1}+\dots+\alpha^{p^2-1} x_{p,p}.\]
For $1/p>\alpha > 0$, there exist a decreasing sequence of vertices of the $p \times p$ Birkhoff-von Neumann polytope of length $Kp!$ for a universal constant $K$.
\end{theorem}

However, Pak (see~\cite{Pak:FourQuestionsBirkhoff}) also showed the more encouraging result that the expected average running time of the simplex method on the Birkhoff polytope with cost vector $c_\alpha$ is $O(p \log p)$.

\subsection{Integer points}


Questions on the integer, or lattice, points of transportation polytopes are very popular in combinatorics. Objects such as magic squares, magic labelling of graphs and sudoku arrangements can be presented as lattice points of transportation polytopes. See, for instance, \cite{Beck:NumMagic,DeLoera:ManyAspects,Stanley:Magic-labelings} and the references therein. How many ways are there to fill the entries of a $p \times q$ table with margins $u$ and $v$ using only non-negative \emph{integer} entries $x_{i,j}$? E.g., see Figure~\ref{figure:count-lattice-points}. 
This counting problem is a $\# P$-complete problem, even for $2 \times q$ tables (see~\cite{Dyer:Counting-SharpP}). 
\begin{figure}[hbt]
\begin{center}
\[
\xy
(3,-3)*+{68};
(9,-3)*+{119};
(15,-3)*+{26};
(21,-3)*+{7};
(3,-9)*+{20};
(9,-9)*+{84};
(15,-9)*+{17};
(21,-9)*+{94};
(3,-15)*+{15};
(9,-15)*+{54};
(15,-15)*+{14};
(21,-15)*+{10};
(3,-21)*+{5};
(9,-21)*+{29};
(15,-21)*+{14};
(21,-21)*+{16};
(3,-27)*+{108};
(9,-27)*+{286};
(15,-27)*+{71};
(21,-27)*+{127};
(28,-3)*+{220};
(28,-9)*+{215};
(28,-15)*+{93};
(28,-21)*+{64};
(0,0); (24,0) **\dir{-}; 
(0,-6); (24,-6) **\dir{-}; 
(0,-12); (24,-12) **\dir{-}; 
(0,-18); (24,-18) **\dir{-}; 
(0,-24); (24,-24) **\dir{-}; 
(0,0); (0,-24) **\dir{-}; 
(6,0); (6,-24) **\dir{-}; 
(12,0); (12,-24) **\dir{-}; 
(18,0); (18,-24) **\dir{-}; 
(24,0); (24,-24) **\dir{-}; 
(12,-33)*+{\text{There are 1,225,914,276,768,514 such tables.}};
\endxy
\]
\end{center}
\caption{This transportation polytope has many lattice points!}\label{figure:count-lattice-points}
\end{figure}


The lattice points of dilations of the Birkhoff polytope are called \emph{semi-magic squares}: that is to say, a semi-magic square is an integral lattice point in a transportation polytope where every row and column sum is the same, namely $\zeta$. The number $\zeta$ is called the \emph{magic number}.
Counting these objects is a rather natural combinatorial problem that has been studied by many researchers. 
In~\cite{De-Loera:A-generating-function}, De~Loera, Liu, and Yoshida presented a generating function for the number of semi-magic squares and formulas for the coefficients of the Ehrhart polynomial of the $p$th Birkhoff polytope $B_p$. In particular they also
deduced a combinatorial formula for the volume of Birkhoff polytopes. The volume formula is a multivariate generating function for the lattice points of the Birkhoff polytope and all its dilations. Unfortunately the number of terms, which alternate in sign, is quite large. The 
summation runs over all the possible arborescences of a complete graph in $p$ nodes ($p^{p-2}$ of them) and the $p!$ permutations, thus the formula is quite large and not efficient to evaluate.  The key elements of this formula come from understanding triangulations
of the tangent cones of the Birkhoff polytope and the algorithmic theory of lattice points developed by Barvinok (see~\cite{Barvinok:PolyTimeIntPoints,Barvinok:ShortRationalGenerating}). More recently Liu (see~\cite{liu:perturbation}) described the same kind of generating functions for perturbations of the Birkhoff
polytope into simple transportation polytopes (i.e., the margin sum conditions are not one but a small change in value). She obtained 
similar combinatorial formulas for the generalized Birkhoff $kp \times p$ polytope. She also recovered the formula for the maximum possible number of vertices of transportation polytopes of order $kp \times p$ that had been studied in the literature before.  Prior work on enumeration includes~\cite{Carlitz:EnumerationSymmetricArrays}, where Carlitz described lattice points of dilations of the Birkhoff polytope using exponential generating functions.

Counting magic squares and lattice points in (dilations of) Birkhoff polytopes is related to computing their volumes. 
The computation of volumes and triangulations of the Birkhoff polytope is related to the problem of generating a random doubly stochastic matrix (see~\cite{Chan:VolumePolytope}).
The volume problem has been studied by many researchers (see~\cite{Ahmed:Algebraic-Combinatorics,Ahmed:Polytopes-of-Magic, Beck:NumMagic, Beck:EhrhartBirkhoff, Chan:VolumePolytope, Diaconis:Random-Matrices, Halleck:MagicSquaresDiophantine,Hemmecke:On-the-computation-of-Hilbert,Pak:FourQuestionsBirkhoff, Stanley:LinearHomDioMagicLabel, Stanley:Magic-labelings}, among others).
The exact value of the volume of the $p$th Birkhoff polytope $B_p$ is known (see~\cite{Pixton:The-Volumes-of-Birkhoff}) only up to $p = 10$. Canfield and McKay (see~\cite{Canfield:VolumeBirkhoff}) presented an asymptotic formula for the volume of the $p$th Birkhoff polytope $B_p$. In~\cite{Barvinok:AsymptoticEstimatesNumberTransportation}, Barvinok  also presented asymptotic upper and lower bounds for the volumes of $p \times q$ classical transportation polytopes and the number of $p \times q$ semi-magic rectangles.

The currently known exact values of $a(p)$ are summarized in Table~\ref{table:birkhoff-volume}. 
\begin{table}[hbt]
{\small
\begin{tabular}{|c|l|}
\hline
$p$ & $a(p)$ \\ \hline
1&1 \\ \hline
2&1 \\ \hline
3&3 \\ \hline
4&352 \\ \hline
5&4718075 \\ \hline
6&14666561365176 \\ \hline
7&17832560768358341943028 \\ \hline
8&12816077964079346687829905128694016 \\ \hline
9&7658969897501574748537755050756794492337074203099 \\ \hline
10&5091038988117504946842559205930853037841762820367901333706255223000 \\ \hline
\end{tabular}
}
\caption{Normalized volumes of Birkhoff polytopes}\label{table:birkhoff-volume}
\end{table}

To compute more values it would useful to know the answer to the following problem:

\begin{openproblem}
Is there a short (polynomial time computable) formula for the normalized volume $a(p)$ of the $p \times p$ Birkhoff-von Neumann
polytope? 
\end{openproblem}


Besides knowing the volumes and the number of vertices, we are interested in knowing the so-called \emph{integer range} of a coordinate in a transportation polytope $P$. This asks the following: fixing $i$ and $j$, do all integers in an interval appear as the value of the coordinate 
$x_{i,j}$ of among the set of lattice points of $P$? For classical transportation polytopes, the answer is yes:
\begin{lemma}[Diaconis and Gangolli~\cite{DG}, Integer range of a coordinate]\label{lemma:classical-integer-range}
For an entry $x_{i,j}$ of the transportation polytope with marginals $u$ and $v$, the set of all possible integral values are the integers on a segment.
\end{lemma}
This gives a method of performing the so-called \emph{sequential importance sampling} (see, e.g.,~\cite{Chen:SequentialImportanceSampling,Xi:EstimatingNumberMultiwaytables}). Chen et al.{} (see~\cite{Chen:SequentialImportanceSampling}) use the interval property to justify correctness of their algorithm for the sequential sampling of entries in multi-way contingency tables with given constraints. This method of sampling contingency tables with given margins introduced in~\cite{Chen:SequentialMonteCarlo} is later extended by Chen in~\cite{Chen:ConditionalInference} to sample tables with fixed marginals and a given set of structural zeros. 

For many applications, again including sampling and enumerating lattice points,  we are interested in having  a set of  ``local moves'' or operations that connect the set of \emph{all} integer contingency tables with fixed margins. E.g.,  such a set of moves is important in probability and statistics in the interest of running Markov chains on contingency tables (see~\cite{DO1}). As it turns out the set of moves necessary is quite simple:

\begin{lemma} \label{markovbases2way}
The set of ``rectangular'' vectors whose entries are $0$,$-1$, and $1$ (as in Figure~\ref{figure:graver-basis})  corresponding to 
4-cycles in the complete bipartite graph $K_{p,q}$, with a $1$ and a $-1$ in each row column are integer vectors in the kernel 
of the constraint matrix of $2$-way transportation polytopes. They are simple moves that connect all lattice points of any 
$2$-way transportation polytope.
\end{lemma}
\begin{figure}[htbp]
\begin{center}
\begin{tabular}{|c|c|c|c|c|c|c|c|c|}
\hline
 0 & 0 & 0 & 0 & 0 & 0 & 0 & 0 & 0\\\hline
 0 & 0 &-1 & 0 & 0 & 0 & 1 & 0 & 0\\\hline
 0 & 0 & 0 & 0 & 0 & 0 & 0 & 0 & 0\\\hline
 0 & 0 & 0 & 0 & 0 & 0 & 0 & 0 & 0\\\hline
 0 & 0 & 1 & 0 & 0 & 0 &-1 & 0 & 0\\\hline
 0 & 0 & 0 & 0 & 0 & 0 & 0 & 0 & 0\\\hline
 0 & 0 & 0 & 0 & 0 & 0 & 0 & 0 & 0\\\hline
\end{tabular}
\end{center}
\caption{Typical monomial in Graver basis}\label{figure:graver-basis}
\end{figure}

Using these moves one can run a Markov chain on all the vertices of a transportation polytope, where we move from one vertex to another
by adding one of the randomly generated moves that preserves non-negativity. Cryan et al. (see~\cite{CDMS}) have shown that
the associated Markov chain mixes rapidly when the number $p$ or $q$ of rows or columns is assumed fixed.

This set of vectors is an example of a {\em Graver basis} for the kernel of the matrix associated to the $2$-way transportation polytope
in question (see~\cite{Graver:GraverBasis}).  Formally, to define a Graver basis, we first describe a partial order $\sqsubseteq$ on $\mathbb{Z}^n$. Given two integer vectors $u,v \in \mathbb{Z}^n$, we say $u \sqsubseteq v$ if $|u_k| \leq |v_k|$ and $u_kv_k \geq 0$ for all $k = 1,\dots,n$. Then the \emph{Graver basis} of a matrix $A$ is the set of all $\sqsubseteq$-minimal vectors in $\{x \in \mathbb{Z}^n \mid Ax = 0, x \not= 0\}$. Graver bases are
quite important in optimization (see Chapters 3 and 4 of \cite{alggeoopt} and the nice book \cite{Onn:NonlinearDiscreteOptimization} for details).

In the next section, we discuss multi-way transportation polytopes. As we will see, their behavior is much more complicated.

\section{Multi-way transportation polytopes}\label{section:multiway}

Classical transportation polytopes were called \emph{$2$-way transportation polytopes} because the coordinates $x_{i,j}$ have two indices. We can consider generalizations of $2$-way transportation polytopes by having coordinates indexed by three or more integers (e.g., $x_{i,j,k}$ or $x_{i,j,k,l}$).  As the number of indices grows the possible form and shape of constraints grows. There has been very active work on understanding
the corresponding polyhedra (see e.g., ~\cite{Haley:The-multi-index-problem,
Haley:Note-on-the-Letter, Moravek:On-the-necessary-conditions, Moravek:On-Necessary-ConditionsClass,
 Schell:Distribution-of-a-product,  Smith:Further-necessary, Smith:A-Procedure-for-Determining, Smith:On-the-Moravek-and-Vlach, Vlach:SolutionsPlanarTransportation}). As we will see here the case of $3$-way transportation problems, i.e., three indices,
is already so complicated that in a sense contains all polyhedral geometry and combinatorial optimization!

A {\em $d$-way table} of size $p_1 \times \dots \times p_d$ is a $p_1 \times p_2 \times \dots \times p_d$  array of non-negative
real numbers $x=(x_{i_1,\dots,i_d})$, $1\leq i_\ell\leq p_\ell$. Given an integer $m$, with  $0\leq m <d$, an {\bf $m$-margin} of the
$d$-way table $x$ is one of the $\binom{d}{m}$ possible $m$-tables obtained by summing the entries over all but $m$ indices.
For example, if $(x_{i,j,k})$ is a $3$-way table then its $0$-marginal is $x_{+,+,+}=\sum_{i=1}^{p_1}\sum_{j=1}^{p_2}\sum_{k=1}^{p_3}
x_{i,j,k}$, it has  three $1$-margins, which are $x_{i,+,+})=\sum_{j=1}^{p_2}\sum_{k=1}^{p_3}x_{i,j,k}$ and likewise
$(x_{+,j,+})$, $(x_{+,+,k})$. Finally $x$ has three $2$-margins given by the sums
$(x_{i,j,+})=\sum_{k=1}^{p_3} x_{i,j,k}$ and likewise $(x_{i,+,k})$, $(x_{+,j,k})$.

A {\em $d$-way transportation polytope} of size $p_1 \times \dots \times p_d$ defined by $m$-marginals is the set of all $d$-way tables of size $p_1 \times p_2 \times \dots \times p_d$ with the specified marginals. When $d=2$, we recover the classical transportation polytopes of the previous section. When $d \geq 3$, the transportation polytope is also called a \emph{multi-way transportation polytope}. When $d=3$, we will typically denote the size of the transportation polytope by $p \times q \times s$ instead of $p_1 \times p_2 \times p_3$.

In a well-defined sense the most important margins of a $d$-way transportation polytope are the $(d-1)$-margins:

\begin{theorem}[Junginger~\cite{Junginger:MultidimensionalTransportationProblem}] 
There exists a polynomial
time algorithm that, given a linear (integer)
minimization problem over a $d$-way $p_1\times \dots \times p_d$
transportation polytope $T_{d,m}$ with fixed $m$-marginals and cost vector $c$,
computes an associated linear functional $\hat{c}$ and a
$d$-way $(p_1+1)\times \dots \times (p_d+1)$ transportation polytope
$T_{d,d-1}$ with fixed $(d-1)$-marginals such that if $y$
is an optimal (integral) solution for $T_{d,d-1}$ its entries
with indices with the original range also give an optimal (integral) 
solution of $T_{d,m}$.
\end{theorem} 

\begin{example}
We illustrate Junginger's theorem in the $3$-way case. 
Suppose we have a linear optimization problem over a
$3$-way $p \times q \times s$ transportation defined by $1$-marginals:
\[
\begin{array}{ll}
\text{minimize} 
        & \displaystyle \sum_{i_1=1}^p \sum_{i_2=1}^q \sum_{i_3=1}^s  c_{i_1,i_2,i_3}x_{i_1,i_2,i_3} \\
\text{subject to}
        & 
        \left\{
        \begin{array}{l}
\displaystyle \sum_{i_2=1}^q \sum_{i_3=1}^s x_{i_1,i_2,i_3}=b_{i_1,+,  +},\\
\displaystyle \sum_{i_1=1}^p \sum_{i_3=1}^s x_{i_1,i_2,i_3}=b_{+,  i_2,+},\\
\displaystyle \sum_{i_1=1}^p \sum_{i_2=1}^q x_{i_1,i_2,i_3}=b_{+,  +,  i_3},\\
x_{i_1,i_2,i_3} \geq 0.
        \end{array}
        \right.
\end{array}
\]

Junginger showed this can be solved instead using a $3$-way $(p+1) \times (q+1) \times (s+1)$ 
transportation polytope with fixed $2$-marginals:
\[
\begin{array}{ll}
\text{minimize} 
        & \displaystyle \sum_{i_1=1}^{p+1} \sum_{i_2=1}^{q+1} \sum_{i_3=1}^{s+1}  \hat{c}_{i_1,i_2,i_3}y_{i_1,i_2,i_3} \\
\text{subject to}
        & 
        \left\{
        \begin{array}{l}
\displaystyle \sum_{i_1=1}^{p+1} y_{i_1,i_2,i_3}=a_{+,i_2,i_3},\\
\displaystyle \sum_{i_2=1}^{q+1} y_{i_1,i_2,i_3}=a_{i_1,+,i_3},\\
\displaystyle \sum_{i_3=1}^{s+1} y_{i_1,i_2,i_3}=a_{i_1,i_2,+},\\
y_{i_1,i_2,i_3} \geq 0.
        \end{array}
        \right.
\end{array}
\]
Here the cost coefficients $\hat{c}_{i_1,i_2,i_3}$ and the $2$-marginals $a$ are as follows:
\[
\hat{c}_{i_1,i_2,i_3}= \left\{
\begin{array}{ll}
c_{i_1,i_2,i_3}, & \mbox{if all 3 indices are within the original ranges} \\
M, & \mbox {if exactly 2 of the indices are within the original range} \\
0, & \mbox{otherwise.} \cr
\end{array}
\right.
\]
Let $\beta=\max(b_{i_1,+,+},b_{+,i_2,+},b_{+,+,i_3})$.

When $i_1,i_2,i_3$ stay within the original ranges:
\[ a_{+,i_2,i_3}=\beta; \quad a_{i_1,+,i_3}=\beta; \quad a_{i_1,i_2,+}=\beta. \]
When we go outside the ranges in exactly one of the indices:
\[a_{+,q+1,i_3}=p\beta- b_{+,+,i_3}, \quad a_{+,i_2,s+1}=p\beta -b_{+,i_2,+},\]
\[a_{i_1,+,s+1}=q\beta- b_{i_1,+,+}, \quad a_{p+1,+,i_3}=q\beta -b_{+,+,i_3},\]
\[a_{p+1,i_2,+}=s\beta- b_{+,i_2,+}, \quad a_{i_1,q+1,+}=s\beta -b_{i_1,+,+}.\]
Finally, when exactly two of the indices are outside the original range:
\[a_{+,q+1,s+1}=a_{p+1,+,s+1}=a_{p+1,q+1,+}=\beta.\]

Now for each solution $x_{i_1,i_2,i_3}$ of the
$3$-way problem with $1$-marginals we can recover a {\bf unique} solution
$y_{i_1,i_2,i_3}$ of the $3$-way problem with $2$-marginals that has the
same objective function value plus a constant. If $x$ is integral,
then $y$ will be integral too when the marginals have integral
entries. For this set the value of $y_{i_1,i_2,i_3}:=x_{i_1,i_2,i_3}$ when all $i_\ell$ are in the
original range. Using the new $2$-marginal equations determine the values of those
variables $y_{i_1,i_2,i_3}$ with exactly \emph{one} index outside original range. 
Thus for fixed $i_2,i_3$ in the original range:
\[y_{p+1,i_2,i_3}=a_{+,i_2,i_3}-\sum_{i_1=1}^p y_{i_1,i_2,i_3} =\beta \sum_{i_1=1}^p x_{i_1,i_2,i_3} \geq 0.\]
Next fill the values of those variables $y_{i_1,i_2,i_3}$ with exactly \emph{two} indices outside range. Finally fill the variable $y_{p+1,q+1,s+1}$.
It is easy (but tedious) to check that $y_{i_1,i_2,i_3}$ is indeed feasible in the $2$-marginal problem.

Now the objective function value is
\[\sum_{i_1=1}^{p+1} \sum_{i_2=1}^{q+1} \sum_{i_3=1}^{s+1} \hat{c}_{i_1,i_2,i_3}y_{i_1,i_2,i_3}= \sum_{i_1=1}^{p} \sum_{i_2=1}^{q} \sum_{i_3=1}^{s}
 c_{i_1,i_2,i_3}x_{i_1,i_2,i_3}\ + \]
\[M \left(\sum_{i_1=1}^p x_{i_1,q+1,s+1} + \sum_{i_2=1}^q x_{p+1,i_2,s+1} +\sum_{i_3=1}^s x_{p+1,q+1,i_3}\right)\]
which is equal to
\[\sum_{i_1=1}^{p} \sum_{i_2=1}^{q} \sum_{i_3=1}^{s}
 c_{i_1,i_2,i_3}x_{i_1,i_2,i_3}\ +3M\beta.\]

Conversely, if $y$ is the optimal solution for
the $2$-marginal problem, the restriction $x$ to those variables with indices 
$i_1 \leq p, i_2 \leq q, i_3 \leq s$ is an optimal solution of the $1$-marginal
problem. For this note that because $y$ is optimal the entries of variables with
two indices above the original range (e.g. $y_{p+1,q+1,i_3}$) must
be zero because their cost is $M$ (a huge constant). Next check $x_{i_1,i_2,i_3}$ is feasible for the $1$-marginal problem. 
Non-negativity is easy: note that
\[\sum_{i_2=1}^q y_{p+1,i_2,i_3}=\sum_{i_2=1}^q \left(a_{+,i_2,i_3}-\sum_{i_1=1}^p y_{i_1,i_2,i_3}\right) = q\beta-b_{+,+,i_3}.\]
Therefore, for the $1$-marginal $b_{+,+,i_3}$,
\[ \sum_{i_1=1}^p \sum_{i_2=1}^q y_{i_1,i_2,i_3}= \sum_{i_1=1}^{p+1} \sum_{i_2=1}^q y_{i_1,i_2,i_3} - \sum_{i_2=1}^q y_{p+1,i_2,i_3} = q\beta -(q\beta -b_{+,+,i_3})= b_{+,+,i_3}.\]
and the same can be checked for other $1$-marginals.
\end{example}

Depending on the application a transportation problem may have a combination of margins that define polyhedron.

For $3$-way transportation problems there are two natural generalization of $2$-way transportation polytopes to $3$-way transportation polytopes, whose feasible points are $p \times q \times s$ tables of non-negative reals satisfying certain sum conditions:
\begin{itemize}
\item First, consider the $3$-way transportation polytope of size $p \times q \times s$ defined by \emph{$1$-marginals}: Let $u=(u_1,\dots,u_p) \in \R^p$,
$v=(y_1,\dots,y_q) \in \R^q$, and $w=(w_1,\dots,w_s) \in \R^s$ be three vectors. Let $P$ be the polyhedron defined by the following $p+q+s$ equations in the $pqs$ variables $x_{i,j,k} \in \R_{\geq 0}$ ($i\in[p], j\in[q], k\in[s]$):
\begin{equation}\label{equation:1marginals}
\sum_{j,k} x_{i,j,k} = u_{i}, \forall i \quad 
\sum_{i,k} x_{i,j,k} = v_{j}, \forall j \quad 
\sum_{i,j} x_{i,j,k} = w_{k}, \forall k.
\end{equation}
In~\cite{YKK}, $3$-way transportation polytopes defined by all $1$-marginals are known as \emph{$3$-way axial transportation polytopes}.

\item Similarly, a $3$-way transportation polytope of size $p \times q \times s$ can be defined by specifying three real-valued matrices $U$, $V$, and $W$ of sizes $q \times s$, $p \times s$, and $p \times q$ (respectively). These three matrices specify the line-sums resulting from fixing two of the indices of entries and adding over the remaining index. That is to say, the polyhedron $P$ is defined by the following $pq+ps+qs$ equations, the \emph{$2$-marginals}, in the $pqs$ variables $x_{i,j,k} \in \R_{\geq 0}$ satisfying:
\begin{equation}\label{equation:2marginals}
\sum_{i} x_{i,j,k} = U_{j,k}, \forall j,k\quad 
\sum_{j} x_{i,j,k} = V_{i,k}, \forall i,k\quad 
\sum_{k} x_{i,j,k} = W_{i,j}, \forall i,j.
\end{equation}
In~\cite{YKK}, the $3$-way transportation polytopes defined by $2$-marginals are called \emph{$3$-way planar transportation polytopes}.
\end{itemize}

\subsection{Why $d$-way transportation polytopes are harder}

The $3$-way transportation polytopes are very interesting because of the following universality theorem of De~Loera and Onn in~\cite{DO2} which says that for any rational convex polytope $P$, there is a $3$-way planar transportation polytope $T$ isomorphic to $P$ in a
very strong sense. 

We say a polytope $P\subset{\mathbb R}^p$ is {\em representable} as a polytope
$T\subset{\mathbb R}^q$ if there is an injection $\sigma:\{1,\dots,p\}\longrightarrow \{1,\dots,q\}$  such that the projection
$\pi:{\mathbb R}^q\longrightarrow{\mathbb R}^p$
\[x=(x_1,\dots,x_q)\mapsto \pi(x)=(x_{\sigma(1)},\dots,x_{\sigma(p)})\]
is a bijection between $T$ and $P$ and between the sets of integer
points $T\cap{\mathbb Z}^q$ and $P\cap{\mathbb Z}^p$.

Note that if $P$ is representable as $T$ then $P$ and $T$ have same facial
structure and all linear or integer programming programs are
polynomial-time equivalent. We can state the universality result as follows:

\begin{theorem}[Universality~\cite{DO2}] \label{theorem:universality-with-complexity}
Any polytope $P=\{y\in{\mathbb R}_{\geq 0}^n:Ay=b\}$
with integer $m\times n$ matrix $A=(a_{i,j})$ and integer vector $b$
is polynomial-time representable as a slim  $r \times c \times 3$ transportation polytope
\[T\quad=\quad\left\{\,x\in{\mathbb R}_{\geq 0}^{r\times c\times 3}\ :\, 
\sum_i x_{i,j,k}=U_{j,k}\,,\
\sum_j x_{i,j,k}=V_{i,k}\,,\ 
\sum_k x_{i,j,k}=W_{i,j}\,\right\}\ ,\]
with $r=O(m^2(n+L)^2)$ rows and $c=O(m(n+L))$ columns,
where $L:=\sum_{j=1}^n\max_{i=1}^m \lfloor \log_2|a_{i,j}| \rfloor$.
\end{theorem}
The constructive proof of Theorem \ref{theorem:universality-with-complexity} follows three steps.
\begin{enumerate}
\item Decrease the size of the coefficients used in the constraints.
\item Encode the polytope $P$ as a transportation polytope with $1$-margins and with some entries bounded
\item Encode any transportation polytope with $1$-margins and bounded entries into a new transportation polytope with $2$-margins
\end{enumerate}

We only explain steps 1 and 2 which already give an interesting corollary.

{\bf Step 1:} Given $P=\{y\geq 0:Ay=b\}$ where $A=(a_{i,j})$ is an integer matrix
and $b$ is an integer vector. We represent it as a polytope $Q=\{x\geq
0:Cx=d\}$, in polynomial-time, with a $\{-1,0,1,2\}$-valued matrix
$C=(c_{i,j})$ of coefficients. For this use the binary expansion
$|a_{i,j}|=\sum_{s=0}^{k_j}t_s 2^s$ with all $t_s\in\{0,1\}$, we
rewrite this term as $\pm\sum_{s=0}^{k_j}t_s x_{j,s}$.

For example, the equation $3y_1-5y_2+2y_3=7$ becomes
\[
\begin{array}{rrrrrrrcl}
2x_{1,0} & -x_{1,1} &          &          &          &          &         &=&0,\\                                                                                
         &          & 2x_{2,0} & -x_{2,1} &          &          &         &=&0,\\                                                                                
         &          &          & 2x_{2,1} & -x_{2,2} &          &         &=&0,\\                                                                                
         &          &          &          &          & 2x_{3,0} & -x_{3,1}&=&0,\\                                                                                
 x_{1,0} & +x_{1,1} & -x_{2,0} &          & -x_{2,2} &          & +x_{3,1}&=&7.
\end{array}
\]

{\bf Step 2:} Here is a sketch.  Each equation $k=1,\dots,m$ will be
encoded in a ``horizontal table'' plus an extra layer of ``slacks''.
Each variable $y_j$, $j=1,\dots,n$ will be encoded in a ``vertical
box''. Other entries are zero. See Figure~\ref{figure:universal-step-2}.
\begin{figure}[hbt]
\begin{center}
\includegraphics[width=4 cm]{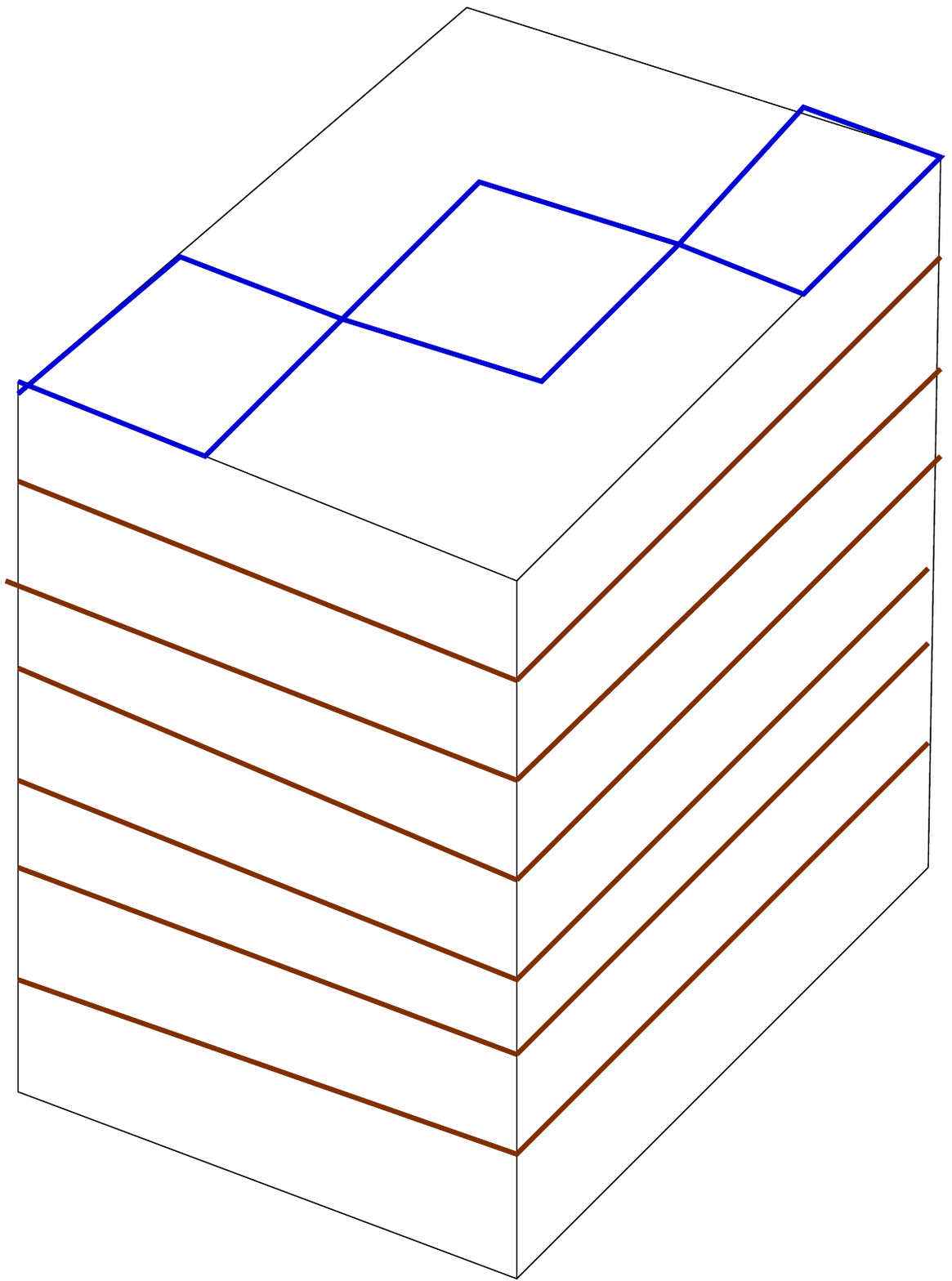}
\end{center}
\caption{Each equation is encoded in one of the $6$ horizontal tables, with the seventh table used for slacks}\label{figure:universal-step-2}
\end{figure}
Given $P=\{y\geq 0:Ay=b\}$ where $A=(a_{i,j})$ is an $m\times n$
integer matrix and $b$ is an integer vector: we assume that $P$ is
bounded and hence a polytope, with an integer upper bound $U$ on the
value of any coordinate $y_j$ of any $y\in P$.
The sizes of the layers will be given by the numbers
$r_j\quad:=\quad \max \left(\sum_k\{a_{k,j}:a_{k,j}>0\}\,,\,\sum_k\{|a_{k,j}|:a_{k,j}<0\}\right)$
and $r:=\sum_{j=1}^n r_j$, $R:=\{1,\dots,r\}$, $m+1$ and $H:=\{1,\dots,m+1\}$.
Each equation $k=1,\dots,m$ is encoded in a ``horizontal table''
$R\times R\times\{k\}$. The last horizontal table $R\times R\times\{m+1\}$ is included
for consistency and its entries can be regarded as ``slacks".
Each variable $y_j$, $j=1,\dots,n$ will be encoded in a ``vertical box"
$R_j\times R_j \times H$, where $R=\biguplus_{j=1}^n R_j$ is the natural
partition of $R$ with $|R_j|=r_j$, namely with $R_j:=\{1+\sum_{l<j}r_l,\dots,\sum_{l\leq j}r_l\}$.

For instance, if we have three variables,
with $r_1=3, r_2=1, r_3=2$ then $R_1=\{1,2,3\}, R_2=\{4\}, R_3=\{5,6\}$,
and the top view of the matrix $x=(x_{i,j,+})$ is

\[
\left(\begin{matrix}
 x_{1,1,+} & x_{1,2,+} & 0 & 0 & 0 & 0 \cr
 0 & x_{2,2,+} & x_{2,3+} & 0 & 0 & 0 \cr
 x_{3,1,+} & 0 & x_{3,3,+} & 0 & 0 & 0 \cr
 0 & 0 & 0 & x_{4,4,+} & 0 & 0 \cr
 0 & 0 & 0 & 0 & x_{5,5,+} & x_{5,6,+} \cr
 0 & 0 & 0 & 0 & x_{6,5,+} & x_{6,6,+} \cr
\end{matrix} \right)
 \ =\ \left(\begin{matrix}
 y_1 & {\bar y}_1 & 0 & 0 & 0 & 0 \cr
 0 & y_1 & {\bar y}_1 & 0 & 0 & 0 \cr
 {\bar y}_1 & 0 & y_1 & 0 & 0 & 0 \cr
 0 & 0 & 0 & U & 0 & 0 \cr
 0 & 0 & 0 & 0 & y_3 & {\bar y}_3 \cr
 0 & 0 & 0 & 0 & {\bar y}_3 & y_3 \cr
\end{matrix}\right).
\]

The actual vertical position is decided with respect to the equations
that contain a variable. Now, to specify the actual $1$-margins: All
``vertical'' plane-sums are set to the same value $U$, that is,
$u_j:=v_j:=U$ for $j=1,\dots,r$. All entries not in the union
$\biguplus_{j=1}^n R_j\times R_j\times H$ of the variable boxes will
be forbidden. The horizontal plane-sums $w$ are determined as
follows: For $k=1,\dots,m$, consider the $k$th equation $\sum_j
a_{k,j}y_j=b_k$.  Define the index sets $J^+:=\{j:a_{k,j}>0\}$ and
$J^-:=\{j:a_{k,j}<0\}$, and set $w_k:=b_k+U\cdot\sum_{j\in
J^-}|a_{k,j}|$.

\begin{example}
What do the three steps of this construction do, if one
starts with the zero-dimensional polytope  $P=\{y \mid 2y=1, \, y\geq 0\}$? In this case, we obtain the $2$-margins of a $3$-way transportation polytope shown in Figure~\ref{figure:universality0dim}.
\begin{figure}[hbt]
\begin{center}
\includegraphics[width=8 cm]{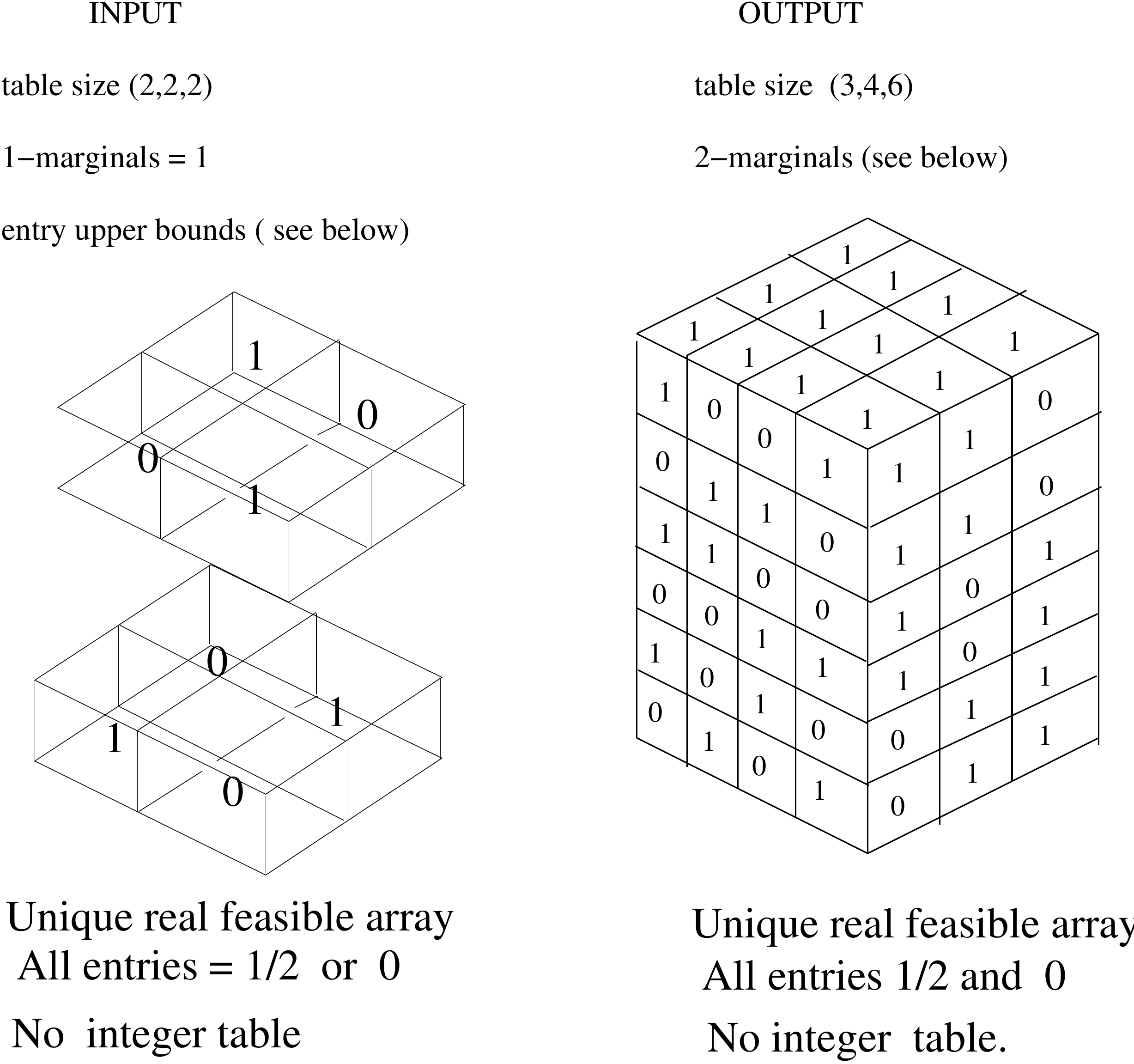}
\end{center}
\caption{Illustration of the universality theorem: A polyhedron consisting of the single point $y=\frac{1}{2}$ is represented
by a $3 \times 4 \times 6$ polytope with 2-margins shown at the right of the figure. The left side of the figure shows the 
encoding using only the first two steps of the algorithm.}\label{figure:universality0dim}
\end{figure}
\end{example}

\subsection{Comparing $2$-way and $3$-way transportation polytopes}

We want to stress some consequences of the construction. First of all, simply from the first two steps above
the following interesting theorem follows. Any rational polyhedron is a face of some axial $3$-way transportation polytope.

\begin{corollary}\label{corollary:1marginface}
Any rational polytope
$P=\{ y \in {\mathbb R}^n \mid Ay=b, \, y \geq 0 \}$ is polynomial-time
representable as {\bf a face} of a $3$-way $r \times c \times 3$
transportation polytope with $1$-margins
\[T=\left\{\,x\in {\mathbb R}_{\geq 0}^{r\times c\times 3}\ :\ 
\sum_{j,k} x_{i,j,k}=u_{i}, 
\,
\sum_{i,k} x_{i,j,k}=v_{j},
\,
\sum_{i,j} x_{i,j,k}=w_{k}
\right\}\ .\]
\end{corollary}

The properties we have seen in Section~\ref{section:classical} for  classical $2$-way transportation polytopes 
raise the issue whether the analogous questions or properties hold for multi-way transportation polytopes. We will address the following:

\begin{enumerate}

\item {\bf Real feasibility (Vlach Problems~\cite{Vlach:SolutionsPlanarTransportation}):} 
Is there a simple characterization in terms of the $2$-margins of those $3$-way
transportation polytopes which are empty?

In particular, do any of the conditions on the margins proposed by Schell,
Haley, Moravek and Vlack (see pages 374--376 of~\cite{YKK}) suffice to guarantee that the polytope is non-empty?

\item {\bf Dimension:} What are the possible dimensions of a $p \times
q \times s$ transportation polytope? Is it always equal to
$(p-1)(q-1)(s-1)$?

\item {\bf Graphs of $3$-way transportation polytopes:}
Do we have a good bound for the diameter?
Is the Linear Hirsch Conjecture true in this case?

\item {\bf Number of vertices of $3$-way transportation polytopes:}
Can one estimate minimum and maximum number of vertices possible?
Do they have a nice characterization?

\item {\bf Integer Feasibility Problem:} Given a prescribed
collection of integral margins that seem to describe a
$d$-way transportation polytope of size $p_1 \times \dots \times p_d$, 
does there exist an integer table with these margins?
Can such an integral $d$-way table be efficiently determined?

\item {\bf Integer Range Property:} 
Given a collection of margins coming from $d$-way table, and an index tuple $(i_1,\dots,i_d)$, do all integer values inside
the range of an interval appear for the coordinate $x_{i_1,\dots,i_d}$ in the corresponding transportation polytope?

\item {\bf Graver/Markov basis for $3$-way transportation polytopes:}
Are the Graver bases for $3$-way transportation polytopes as nice as they are for $2$-way transportation polytopes?
Do $3$-way transportation polytopes have the ``interval property'' for entry values?

\end{enumerate}

Most of these questions had easy solutions for classical transportation polytopes. In the next sections we answer all these questions for multi-way transportation polytopes.

\subsubsection{Feasibility and dimension revisited}


Recall Lemma~\ref{lemma:nonemptycriterion} for $2$-way transportation polytopes, which gave a simple characterization for a non-empty polytope in terms of its margins. From the equations in~\eqref{equation:1marginals}, a similar necessary and sufficient condition for the $3$-way axial transportation polytope to be non-empty can be proved:
\begin{lemma}\label{lemma:axialfeasibility}
Let $P$ be the $p \times q \times s$ axial $3$-way transportation polytope defined by the marginals $u$, $v$, and $w$. The polytope $P$ is non-empty if and only if 
\begin{equation}\label{equation:axialsizigies}
\sum_{i=1}^p u_i = \sum_{j=1}^q v_j = \sum_{k=1}^s w_k.
\end{equation}
\end{lemma}
The proof of this lemma is like Lemma~\ref{lemma:nonemptycriterion} for $2$-way transportation polytopes, but using a $3$-way analogue of the northwest corner rule algorithm.

While similar statements are true for $d$-way transportation polytopes defined by $1$-marginals, the real feasibility problem does not have a known characterization for $m$-marginals with $m \geq 2$, even for $3$-way transportation polytopes. This study is called \emph{real feasibility} or the Vlach Problems (see~\cite{Vlach:SolutionsPlanarTransportation}). The conditions on the
margins proposed by Schell, Haley, Moravek and Vlack (see~\cite{Vlach:SolutionsPlanarTransportation}) are necessary, but not sufficient to guarantee that the polytopes are non-empty. By the universality theorem one cannot expect a simple characterization (in terms of the $2$-marginals of $3$-way transportation polytope) to decide when they are empty. In fact, due to Theorem~\ref{theorem:universality-with-complexity}, given a prescribed collection of marginals that seem to describe a $d$-way transportation polytope of size $p_1 \times \dots \times p_d$, deciding whether there is an integer table with these margins is an NP-complete problem.


Recall that we had also a nice simple dimension formula for $2$-way transportation polytopes.
As a consequence of Lemma~\ref{lemma:axialfeasibility}, the $3$-way axial transportation polytope 
$P$ defined by~\eqref{equation:1marginals} is completely described by only $p+q+s-2$ independent equations. 
The maximum possible dimension for $p \times q \times s$ transportation polytopes defined by $1$-marginals is $pqs-p-q-s+2$.
For planar $3$-way transportation polytopes, one can see that in fact only $pq + ps + qs - p - q - s + 1$ of the defining equations are linearly independent for feasible systems.  The maximum possible dimension for $p \times q \times s$ transportation polytopes defined by $2$-marginals is $(p-1)(q-1)(s-1)$. Unfortunately, by the universality theorem, the dimension of these polytopes can be any number up to $(p-1)(q-1)(s-1)$.

\subsubsection{Combinatorics of faces revisited}

From the universality theorem one can expect that the $f$-vector and indeed the entire combinatorial properties of any rational polytope
will appear when listing all the $f$-vectors of $3$-way transportation polytopes. Indeed, Shmuel Onn has suggested this as a way to systematically enumerate
combinatorial types of polytopes.  In~\cite{DeLoera:GraphsTP}, there was an experimental investigation of the possible polyhedra that
arise for small $3$-way transportation polytopes. The number of vertices of certain low-dimensional $3$-way transportation polytopes 
have been completely classified:

\begin{theorem}
The possible numbers of vertices of non-degenerate $2
  \times 2 \times 2$ and $2 \times 2 \times 3$ axial transportation
  polytopes are those given in Table \ref{axial_vertices}.

  Moreover, every non-degenerate $2 \times 2 \times 4$ axial transportation
  polytope has between $32$ and $504$ vertices.  Every non-degenerate
  $2 \times 3 \times 3$ axial transportation polytopes has between
  $81$ and $1056$ vertices.  The number of vertices of non-degenerate
  $3 \times 3 \times 3$ axial transportation polytopes is at least $729$.
\end{theorem}

\begin{theorem} The possible numbers of vertices of non-degenerate $2
  \times 2 \times 2$, $2 \times 2 \times 3$, $2 \times 2 \times 4$, $2
  \times 2 \times 5$, and $2 \times 3 \times 3$ planar transportation
  polytopes are those given in Table \ref{planar_vertices}. Moreover, 
  every non-degenerate $2 \times 3 \times 4$ planar transportation 
  polytope has between $7$ and $480$ vertices.

\end{theorem}

\begin{table}[hbtp]
 \font\ninerm=cmr8
 \centerline{\ninerm
 \begin{tabular}{|c|c|c|}
 \hline
\textrm{\qquad Size\qquad} & \textrm{Dimension} & \textrm{Possible numbers of vertices} \cr \hline
$2 \times 2 \times 2$ & $4$ & 8 11 14 \\ \hline
$2 \times 2 \times 3$ & $7$ & 18 24 30 32 36 38 40 42 44 46 48 50 52 54 56 58 60 62 64 66 \cr
                      &     & 68 70 72 74 76 78 80 84 86 96 108 \\ \hline
\end{tabular}
}
\caption{Numbers of vertices possible in non-degenerate axial transportation polytopes} \label{axial_vertices}
\end{table}

  \begin{table}[hbtp]
 \font\ninerm=cmr8
 \centerline{\ninerm
 \begin{tabular}{|c|c|c|}
 \hline
\textrm{\qquad Size\qquad} & \textrm{Dimension} & \textrm{Possible numbers of vertices} \cr \hline         
$2 \times 2 \times 2$ & $1$ & 2 \\ \hline
$2 \times 2 \times 3$ & $2$ & 3 4 5 6 \\ \hline
$2 \times 2 \times 4$ & $3$ & 4 6 8 10 12 \\ \hline
$2 \times 2 \times 5$ & $4$ & 5 8 11 12 14 15 16 17 18 19 20 21 22 23 24 25 26 27 28 29 30 \\ \hline
$2 \times 3 \times 3$ & $4$ & 5 8 9 11 12 13 14 15 16 17 18 19 20 21 22 23 \cr
                      &     & 24 25 26 27 28 29 30 31 32 33 34 35 36 37 38 39 40 \cr
                      &     & 41 42 43 44 45 46 47 48 49 50 51 52 53 54 55 56 57 58 59 \\ \hline
\end{tabular}
}
\caption{Numbers of vertices possible in non-degenerate planar transportation polytopes} \label{planar_vertices}
\end{table}  

Again from Theorem~\ref{theorem:universality-with-complexity}, the graph of every rational convex polyhedron 
will appear as the graph of some $3$-way transportation polytope. In particular if the Hirsch Conjecture is true for 
$3$-way transportation polytopes given by $2$-marginals, then it is true for all rational convex polytopes. By 
Corollary~\ref{corollary:1marginface} the graph of every rational convex polytope is the graph of a face of some 
$3$-way transportation polytope given by $1$-margins. Interestingly, in joint work with Onn and Santos 
(see~\cite{DeLoera:GraphsTP}), we proved a quadratic bound on the diameter of $3$-way transportation polytopes given by $2$-margins:

\begin{theorem}
The diameter of every $p \times q \times s$ axial $3$-way transportation polytope is at most  $2(p+q+s-2)^2$.
\end{theorem}

\begin{openproblem}
Prove or disprove the Linear Hirsch Conjecture for $3$-way axial transportation polytopes: is it true that there is a universal constant $k$ such that the diameter of every $p \times q \times s$ axial $3$-way transportation polytope is at most $k(p+q+s)$?
\end{openproblem}

\subsubsection{Integer points revisited}

Integer feasibility becomes now a truly difficult problem because due to Theorem~\ref{theorem:universality-with-complexity}, all linear and integer programming problems are slim $3$-way transportation problems. In other words, any linear or integer
programming problem is equivalent to one that has a $\{0,1\}$-valued constraint matrix, with exactly three $1$'s per column in the constraint matrix, and depends only on the right-hand side data. 

Again, there is bad news for the integer range property. Unlike Lemma~\ref{lemma:classical-integer-range} for $2$-way transportation polytopes, now the values of a variable $x_{i,j,k}$ in a $3$-way transportation polytope can have integer gaps.
Similarly we saw that $2$-way transportation polytopes have a nice \emph{Graver basis}, which we recall is a minimal set of vectors needed to travel between any pair of integer points in the polytope. Unlike the case of $2$-way transportation problems and as a consequence of Theorem~\ref{theorem:universality-with-complexity}, the coefficients in the entries of a Markov basis for $d$-way transportation polytopes can be arbitrarily large (see~\cite{DO3}), not just $0, -1, 1$ as we saw in Lemma~\ref{markovbases2way}.

Since any integer linear programming problem can be encoded as a slim $3$-way transportation problem, the family of $3$-way transportation polytopes really varied. The very same family of $3$-way transportation problems of $p \times q \times 3$ and specified by $2$-margins
contains subproblems that admit fully polynomial approximation schemes as well as subproblems that do not have arbitrarily close
approximation (unless $NP=P$). For this reason, no purely combinatorial approximation algorithm, i.e., one that does not take into account the $2$-margin values, can be devised. 

As we had for $2$-way transportation polytopes, we have $3$-way Birkhoff polytopes, which are much more complicated:

\begin{definition}\label{definition:generalizedbirkhoffmultiway}
The Birkhoff polytope has the following generalizations in the $3$-way setting, these are the \emph{multi-way assignment polytopes}:
\begin{enumerate}
\item
The \emph{generalized Birkhoff $3$-way axial polytope} is the axial $3$-way transportation $p \times q \times s$ polytope whose $1$-marginals are given by the vectors $u = (qs, \ldots, qs) \in \R^p$, $v = (ps, \ldots, ps) \in \R^q$, and $w = (pq, \ldots, pq) \in \R^s$.
\item
The \emph{generalized Birkhoff $3$-way planar polytope} is the planar $3$-way transportation $p \times q \times s$ polytope whose $2$-marginals are given by the $q \times s$ matrix $U_{j,k}=p$, the $p \times s$ matrix $V_{i,k}=q$, and the $p \times q$ matrix $W_{i,j}=s$.
\end{enumerate}
\end{definition}

There is a vibrant study of  $d$-way assignment polytopes. We refer the readers to the recent book
\cite{Burkard+DellAmico+Martello} about assignment problems. We would simply like to mention some results that show how much
harder it is to work with them versus the $2$-way Birkhoff polytope. First, a now-classical result of Karp about the axial assignment
polytope shows it is much more difficult to optimize over $d$-way assignment polytopes.

\begin{theorem}[Karp~\cite{Karp:Reducibility}] 
The optimization problem
\begin{equation*}
\begin{array}{rl}
\text{maximize/minimize} 
	& \displaystyle \sum c_{i,j,k} x_{i,j,k} \\
\text{subject to}
	& 
	\left\{
	\begin{array}{l}
	\displaystyle \sum_{j,k} x_{i,j,k}=1,\\
	\displaystyle \sum_{i,k} x_{i,j,k}=1,\\
	\displaystyle \sum_{i,j} x_{i,j,k}=1,\\
	x\in {\mathbb Z}_{\geq 0}^{p\times p\times p}
	\end{array}
	\right.
\end{array}
\end{equation*}
is NP-hard. 
\end{theorem}

\begin{theorem}[Crama, Spieksma~\cite{CramaSpieksma:Approximation}]
For the minimization problem above, no polynomial time algorithm can even achieve a constant
performance ratio unless NP=P.
\end{theorem}

There are very interesting ``universality'' results about the coordinates of
 vertices of the generalized assignment problem with $1$-margins.

\begin{definition}
For a vertex $x$ of the $d$-way
$1$-margin assignment problem define its \emph{spectrum} to be the vector
$\operatorname{spectrum}(x)$ with positive and decreasing entries which contains the
values of all entries with repetitions deleted.
\end{definition}
For example, the spectrum of a permutation matrix is always $1$. 

Gromova (see~\cite{Gromova}) gave a characterization of which vectors are in 
the spectrum. Here are some of her results:

\begin{theorem}[Gromova~\cite{Gromova}]
Given a positive decreasing vector $\sigma$ of 
rational numbers its \emph{relation matrix} $R(\sigma)$ is the matrix of
all distinct non-negative integer row vectors $\tau$ such that 
$\sigma \cdot \tau = 1$ 
\begin{enumerate}
\item For any positive decreasing vector $v$ with components less than one
whose relation matrix is not empty and whose columns are linearly
independent and for $k \geq \max (1/v_i)$, there is a vertex of a
$d$-way assignment polytope with $1$-margins with spectrum $v$.
\item Take any positive decreasing vector $v$ with components less than one.
It appears as part of the spectrum of some vertex of a $3$-way assignment
polytope.
\end{enumerate}
\end{theorem}

We now discuss $d$-way assignment polytopes defined by $2$-margins.
Recall the Birkhoff-von Neumann Theorem (Theorem~\ref{theorem:birkhoffvonneumann}) which stated that the $p \times p$ Birkhoff polytope has $p!$ vertices. There is a $3$-way analogue of this result: First, recall the $p \times p \times p$ generalized Birkhoff planar polytope (also called the \emph{$2$-marginals assignment polytope}) is the $3$-way transportation polytope of line sums whose $2$-margins are given by $U_{j,k}=V_{i,k}=W_{i,j}=1$. In the case of $2$-margins one can see that a solution, which is a $3$-way array, has in each planar slice a permutation matrix.
This indicates that there is a bijection between the $0/1$ vertices of the $3$-way
$2$-marginals $p \times p \times p$ assignment polytope and the possible $p \times p$ \emph{latin squares}. Although their number is not known
exactly this is enough to say that the number of $0/1$ vertices of the polytope is bounded below by $(p!)^{2p}/p^{p^2}$.
Recently, Linial and Luria (see~\cite{Linial:VerticesdBirkhoff}) proved that the total number of vertices of the 
$p \times p \times p$ generalized Birkhoff polytope is at least exponential in the number of Latin squares of order $p$.

Again we have a hardness result on the planar $3$-way assignment polytope:

\begin{theorem} (Dyer, Frieze~\cite{Dyer:Planar3dm}) The linear optimization problem 
\begin{equation*}
\begin{array}{rl}
\text{maximize} 
	& \displaystyle\sum c_{i,j,k} x_{i,j,k} \\
\text{subject to}
	& 
	\left\{
	\begin{array}{l}
	\displaystyle\sum_{i} x_{i,j,k}=1,\\
	\displaystyle\sum_{j} x_{i,j,k}=1,\\
	\displaystyle\sum_{k} x_{i,j,k}=1,\\
	x\in {\mathbb Z}_{\geq 0}^{p\times p\times p}\\
	\end{array}
	\right.
\end{array}
\end{equation*}
is in general NP-hard, even when $c_{i,j,k} \in \{0,1\}$. However, when
 $c_{i,j,k}=c_{i,j,l}$ for all $l, k$ then the problem is polynomially solvable.
\end{theorem}

Though this maximization problem is NP-hard, we note that Nishizeki and Chiba (see~\cite{Nishizeki:PlanarGraphs}) showed that a PTAS exists.

\section{Further research directions and more open problems}

There are several fascinating areas of research where tables with prescribed sums of their entries play a role. In this last section
we would like to take a quick look at some of these areas and highlight some very nice open questions.

\subsection{$0$-$1$ tables}

We have seen some results like Birkhoff's theorem on permutation matrices that deal specifically with $0$-$1$ tables. Interesting problems about $0$-$1$ tables appear naturally in  combinatorial representation theory (see~\cite{vallejo1, pakrepresentationtheory}) and number theory (see~\cite{alpersnumthy}). Analyzing the properties of $0$-$1$ tables is a classic area of research in combinatorial matrix 
theory. This field combines techniques from combinatorics and group theory and it is so large we do not even attempt to summarize the results available. The reader should consult the books \cite{Brualdi:CombinatorialApproachMatrixTheory,Brualdi:CombinatorialMatrixTheory}. As a taste of the richness of the field of combinatorial matrix theory
let us just talk about results  known for subpolytopes of the Birkhoff polytope. One can consider permutation polytopes, obtained as the convex hull of some vertices of a Birkhoff polytope. (Note that this notion of permutation polytope is distinct from the permutation polytopes of Billera and Sarangarajan in~\cite{Billera:CombinatoricsPermutationPolytopes}.) In~\cite{Onn:Geometry-Complexity}, Onn analyzed the geometry, complexity and combinatorics of permutation polytopes. In~\cite{Baumeister:On-permutation-polytopes}, Baumeister et al. studied the faces and combinatorial types that appear in small permutation polytopes. Brualdi (see~\cite{Brualdi:Convex-polytopes}) investigated the faces of the convex polytope of doubly stochastic matrices which are invariant under a fixed row and column permutation. The \emph{$p$th tridiagonal Birkhoff polytope} is the convex hull of the vertices of the Birkhoff polytope whose support entries are in $\{(i,j) \in [p] \times [p] \mid |i-j| \leq 1\}$. In~\cite{da-Fonseca:Fibonacci-numbers}, da Fonseca et al. counted the number of vertices of tridiagonal Birkhoff polytopes. In~\cite{Costa:The-number-of-faces}, Costa et al. presented a formula counting the number of faces tridiagonal Birkhoff polytopes. Volumes of permutation polytopes
were studied in \cite{mopaper}.

Costa et al. (see~\cite{Costa:The-diameter-of-the-acyclic}) defined a \emph{$p$th acyclic Birkhoff polytope} to be any polytope that is the convex hull of the set of matrices whose support corresponds to (some subset of) a fixed tree graph's edges (including loops). In~\cite{Costa:Face-counting}, Costa et al. counted the faces of acyclic Birkhoff polytopes. In~\cite{Costa:The-diameter-of-the-acyclic}, Costa et al. proved an upper bound on the diameter of acyclic Birkhoff polytopes, which generalized the diameter result of Dahl in~\cite{Dahl:Tridiagonal-doubly}.

Let the \emph{$p$th even Birkhoff polytope} be the convex hull of the $\frac12p!$ permutation matrices corresponding to even permutations. In~\cite{Cunningham:On-the-even-permutation}, Cunningham and Wang confirmed a conjecture of Brualdi and Liu (see~\cite{Brualdi:The-polytope-of-even}) that the $p$th even Birkhoff polytope cannot be described as the solution set of polynomially many linear inequalities. In~\cite{Hood:Some-facets}, Hood and Perkinson described some of the facets of the even Birkhoff polytope and proved a conjecture of Brualdi and Liu (see~\cite{Brualdi:The-polytope-of-even}) that the number of facets of the $p$th even Birkhoff polytope is not polynomial in $p$. In~\cite{Below:On-a-theorem-of-L.-Mirsky}, von Below showed that the condition of Mirsky given in~\cite{Mirsky:Even-doubly} is not sufficient for determining membership of a point in an even Birkhoff polytope. Cunningham and Wang (see~\cite{Cunningham:On-the-even-permutation}) also investigated the membership problem for the even Birkhoff polytope. In~\cite{Below:Even-and-odd-diagonals}, von~Below and R\'enier described even and odd diagonals in even Birkhoff polytopes. In~\cite{Cho:Convex-polytopes}, Cho and Nam introduced a signed analogue of the Birkhoff polytope.

The $0$-$1$ points of transportation polytopes also have a strong connection to \emph{discrete tomography}, which considers the problem of reconstructing binary images (or finite subsets of objects placed in a lattice) from a small number of their projections. The connection to tables is clear as one can think of the position of the objects in points in a grid as the placement of $0$'s and $1$'s in entries of a table. This is a very active field of research.
See~\cite{Alpers:StabilityErrorCorrection, Brunetti:ReconstructionConstraints, Gardner:DiscreteTomographyDetermination, Gardner:SuccessiveDetermination, Gardner:ComputationalComplexity, Gritzmann:IndexSiegelGrids, Gritzmann:UniquenessDiscreteTomography, Gritzmann:SuccessFailureReconstruction, Gritzmann:ApproxBinaryImages, Gritzmann:AlgorithmicInversionRadon, gaborkubabook1, gaborkubabook2}
 and the references therein.
These reconstruction problems are important in CAT scanner development, electron microscope image reconstruction, and quality control in semiconductor production (see, e.g., \cite{Alpers:StabilityErrorCorrection, Gardner:DiscreteTomographyDetermination,Gardner:SuccessiveDetermination} and the references therein).

In light of this discussion, the following open problem is interesting:
\begin{openproblem}
What is the complexity of counting all $2$-way $0$-$1$ tables for given margins?
\end{openproblem}
In the next section, we discuss what is known about enumerating contingency tables in general. 

\subsection{Enumeration, sampling and optimization}

We have seen that counting contingency tables is quite important in combinatorics and statistics.
In~\cite{DO1}, De Loera and Onn gave a complete description of the computational complexity of existence, 
counting, and entry-security in multi-way table problems. The following theorem summarizes what 
is known about counting (specified in terms of binary encoding or unary encoding of data):

\begin{theorem} The computational complexity of the 
counting problem for integral $3$-way tables of size $p \times q \times s$ with 
$2\leq p\leq q \leq s$ and all $2$-marginals specified is 
provided by the following table:
\vskip .2cm \noindent
\begin{tabular}{|c|c|c|c|c|} \hline
& $p,q,s$ & $p,q$ fixed, & $p$ fixed,     & $p,q,s$ variable \\ 
&  fixed  & $s$ variable  & $q,s$ variable & \\ \hline
unary $2$-marginals   & P & P    & \#PC  & \#PC \\
binary $2$-marginals  & P & \#PC & \#PC  & \#PC \\
\hline
\end{tabular}
\end{theorem}

Using the highly-structured Graver bases of transportation polytopes with special restrictions one can do some polynomial-time
optimization on highly difficult problems: E.g., 
De~Loera, Hemmecke, Onn, and Weismantel (see~\cite{DeLoera:NFold}) proved there is a polynomial time algorithm that, given $s$ and fixing $p$ and $q$, solves integer programming problems of $3$-way transportation polytopes of size $p \times q \times s$ defined by $2$-marginals, over any integer objective. Later on, 
in~\cite{De-Loera:Convex-Integer} De~Loera, Hemmecke, Onn, Rothblum, and Weismantel presented a polynomial oracle-time algorithm to solve convex integer maximization over $3$-way planar transportation polytopes, if two of the margin sizes remain fixed. More recently (in~\cite{Hemmecke:A-polynomial-oracle-time}) Hemmecke, Onn, and Weismantel proved a similar result for convex integer minimization.

\subsection{More open problems on transportation polytopes}

We will also mention some more conjectures and open problems on transportation polytopes, and where applicable, give an update on problems where there are solutions and partial answers. We hope this will help to increase the interest in this subject.

\begin{conjecture}
It is impossible to have $p \times q \times s$ non-degenerate $3$-way transportation polytopes, specified by $2$-margin matrices $U,V,W$, whose number $f_0$ of vertices satisfies the
inequalities $(p-1)(q-1)(s-1)+1 < f_0(M(U,V,W)) < 2(p-1)(q-1)(s-1)$?
\end{conjecture}
This conjecture is true when $p,q,s \leq 3$.

\begin{openproblem}
Is it true that the graph of any
$2$-way $p \times q$ transportation polytope is Hamiltonian?
\end{openproblem}
Hamiltonicity of the graph is known to hold for small values of $p$ and $q$.

\begin{openproblem}
Suppose $\phi_1(p,q),\phi_2(p,q),\dots,\phi_{t_{p,q}}(p,q)$ are all possible values of the number of vertices of $p \times q$ transportation polytopes. Give a formula for $t_{p,q}$.
\end{openproblem}

This is related to the problem of enumerating all triangulations or chambers of a vector configuration.

\begin{conjecture}
All integer numbers between 1 and $p+q-1$,
and only these, are realized as the diameters of $p \times q$ transportation
polytopes.
\end{conjecture}

\begin{openproblem}
What are the possible number of facets
for $3$-way $p \times q \times s$ non-degenerate transportation polytopes 
given by $2$-margins?
\end{openproblem}

\begin{openproblem}
What is the largest possible number of
vertices in a $3$-way $p \times q \times s$ transportation polytope?
\end{openproblem}

Recall the $p \times q \times s$ {\em generalized central transportation polytope} is the $3$-way transportation polytope of line sums whose $2$-margins are given by  the $q \times s$ matrix $U_{j,k}=p$,  the $p \times s$ matrix $V_{i,k}=q$, and 
the $p \times q$ matrix $W_{i,j}=s$.  Yemelichev, Kovalev, Kratsov stated in \cite{YKK} the conjecture that
the generalized central transportation polytope had the largest number of vertices among $3$-way transportation polytopes.
This conjecture was proved to be false in~\cite{DeLoera:GraphsTP}. Here are explicit $2$-marginals for a $3 \times 3 \times 3$ transportation polytope which has more vertices ($270$ vertices) than the generalized central transportation polytope, with only $66$ vertices:
{\tiny
	\begin{table}[hbt]
	\begin{center}
	\begin{tabular}{|ccc| |ccc| |ccc|} \hline
	      164424  &    324745  &    127239 &   163445   &    49395 &     403568  & 184032  &    123585 &     269245 \\ \hline
	      262784  &    601074  &   9369116 &   1151824   &   767866 &    8313284 & 886393  &   6722333 &     935582 \\ \hline
	      149654  &   7618489  &   1736281 &   1609500   &  6331023 &    1563901 & 1854344 &    302366 &    9075926 \\ \hline
	\end{tabular}
	\end{center}
	\end{table}
}

\section*{Acknowlegements}
We thank Raymond Hemmecke, Fu Liu, Shmuel Onn, Francisco Santos, and
Ruriko Yoshida for their suggestions and joint work on transportation
problems.  We are also grateful to them and Andreas Alpers, Matthias
Beck, Steffen Borgwardt, Persi Diaconis, Peter Gritzmann, Igor Pak,
Seth Sullivant, and Ernesto Vallejo for suggestions. The first author
is grateful for the support received from NSF grant DMS-0914107.  The
first author is grateful to the Technische Universit\"at M\"unchen for
the hospitality received during the days while writing this survey.

\end{document}